\documentclass[12pt, reqno]{amsart}
\usepackage{amsmath,amsthm,amstext}
\usepackage{amsthm,amsfonts,amssymb,euscript,mathrsfs,graphics,color,amsmath,amssymb,latexsym,marginnote}

\usepackage{hyperref}
\usepackage{setspace,url}

\usepackage{bm,bbm}

\usepackage{scalerel,stackengine}
\stackMath
\newcommand\reallywidecheck[1]{%
\savestack{\tmpbox}{\stretchto{%
  \scaleto{%
    \scalerel*[\widthof{\ensuremath{#1}}]{\kern-.6pt\bigwedge\kern-.6pt}%
    {\rule[-\textheight/2]{1ex}{\textheight}}
  }{\textheight}%
}{0.5ex}}%
\stackon[1pt]{#1}{\scalebox{-1}{\tmpbox}}%
}

\usepackage[margin=.92in]{geometry}

\newtheorem{theorem}{Theorem}
\newtheorem{lemma}{Lemma}[section]

\newtheorem{proposition}{Proposition}[section]
\newtheorem{remark}{Remark}[section]
\newtheorem{assumption}{Assumption}[section]
\newtheorem{definition}{Definition}[section]
\numberwithin{equation}{section}
\newcommand{\eps}{\epsilon}
\newcommand{\lr}[1]{\langle #1 \rangle}
\newcommand{\n}[1]{\| #1 \|}
\newcommand{\lrH}{\sqrt{H+1}}
\newcommand{\Phiiota}{\Phi_{\iota_1 \iota_2 \iota_3 \iota_4} (k, l, m, n, p)}
\newcommand{\muiota}{\mu_{\iota_1 \iota_2 \iota_3 \iota_4} (k, l, m, n, p)}

\newcommand{\distF}{\widetilde{\mathcal{F}}}
\newcommand{\wD}{\widetilde D}
\newcommand{\wg}{\widetilde g}

\renewcommand{\b}{{\rm b}}

\newcommand{\normm}[2]{\left\|#1\right\|_{L^2_{#2}}}
\newcommand{\norm}[3]{\left\|#1\right\|_{L_{#2}^{#3}}}

\newcommand{\rplus}{{\mathbb{R}^+}}

\newcommand{\ip}[1]{\langle #1 \rangle}

\renewcommand{\part}[1]{{\bf Part #1}}

\renewcommand{\l}{\mathcal{L}}
\newcommand{\reals}{\mathbb{R}}

\renewcommand{\Re}{\text{Re }}
\renewcommand{\Im}{\text{Im }}

\newcommand{\RR}{\ensuremath{\mathbb{R}}}

\newcommand{\ZZ}{\ensuremath{\mathbb{Z}}}

\newcommand{\prtl}{\ensuremath{\partial}}

\newcommand{\calH}{\ensuremath{\mathcal{H}}} 
\newcommand{\calL}{\ensuremath{\mathcal{L}}} 
\newcommand{\calO}{\ensuremath{\mathcal{O}}} 

\newcommand{\sech}{\ensuremath{\text{sech}}}
\newcommand{\R}{\mathbb{R}}
\newcommand{\C}{\mathbb{C}}

{\begin{trivlist} \item[]{\textbf{Proof} }}%
{\hspace*{\fill}$\rule{.3\baselineskip}{.35\baselineskip}$\end{trivlist}}

\definecolor{bluegreen}{rgb}{0.0, 0.3, 0.9}

\newcommand{\cfp}[1]{\color{bluegreen} {\tt [FP: #1]} \color{black}}
\newcommand{\fp}[1]{\color{red}{#1} \color{black}}

\newcommand{\comment}[1]{\vskip.3cm\fbox{%
\parbox{0.93\linewidth}{\footnotesize #1}}
\vskip.3cm}

\begin{document}



\title{Asymptotic stability near the soliton \\ for quartic Klein-Gordon in 1D} 

\author{Adilbek Kairzhan}
\address{University of Toronto, Department of Mathematics, 40 St George Street, Toronto, ON,
M5S 2E4, Canada.}
\email{kairzhan@math.toronto.edu}

\author{Fabio Pusateri}
\address{University of Toronto, Department of Mathematics, 40 St George Street, Toronto, ON,
M5S 2E4, Canada.}
\email{fabiop@math.toronto.edu}

\maketitle
\vspace{-0.3in}
\begin{abstract}
We consider the nonlinear focusing Klein-Gordon equation in $1+1$ dimensions 
and the global space-time dynamics of solutions near the unstable soliton. 
Our main result is a proof of optimal decay, and local decay, for even perturbations of the static soliton 
originating from well-prepared initial data belonging to a subset of the {center-stable manifold}
constructed in \cite{BJ89,KMM21}.
Our results complement those of Kowalczyk-Martel-Mu\~noz \cite{KMM21} 
and confirm numerical results of Bizon-Chmaj-Szpak \cite{BiChSz11}
when considering nonlinearities $u^p$ with $p\geq 4$.
In particular, we provide new information both local and global in space
about asymptotically stable perturbations of the soliton 
under localization assumptions on the data.
\end{abstract}

\maketitle



\setcounter{tocdepth}{1}
\tableofcontents

\medskip
\section{Introduction}

\smallskip
\subsection{The problem}
We consider the focusing quartic Klein-Gordon (NLKG) equation 
\begin{equation}
\label{nlkg}
\partial_t^2 u - \partial_x^2 u + u = u^4,
\end{equation}
where $(t, x) \in \RR \times \RR$ are time and space variables, and $u(t,x) \in \reals$. 
Our results apply to more general power nonlinearities as well but we restrict
our discussion and proof to the case $p=4$ for concreteness; see Remark \ref{remp}.
The Hamiltonian associated with \eqref{nlkg} is
\begin{equation}\label{nlkgHam}
\calH(u,u_t) = \frac{1}{2} \int_\reals \left[ (\partial_t u)^2 + (\partial_x u)^2 + u^2 \right]dx 
  - \frac{1}{5}\int_\reals u^5 dx,
\end{equation}
and for initial data in the energy space, $(u(0),u_t(0)) \in H^1(\R)\times L^2(\R)$,
the equation is locally well-posed, and globally well-posed for small solutions.
{Global existence for small energy initial data follows from standard energy estimates, see for example \cite{Caz85}. It is also known that finite time blow-up occurs for solutions with negative energy initial data \cite{Levine} and for large positive initial data \cite{Wang, YX18}, which follows from a convexity argument.}

\eqref{nlkg} admits an explicit time-independent static soliton solution
\begin{equation}
\label{soliton}
Q(x) = (\alpha+1)^{\frac{1}{2 \alpha}} \sech^{1/\alpha} (\alpha x), \quad 
  \alpha := \frac{3}{2}.
\end{equation}
In this paper we study the dynamics of certain classes 
of global-in-time solutions ($t\geq 0$) in the vicinity of $Q$
and, in particular, their behavior on the whole real line and their decay towards the soliton.
Letting $u = Q + v$ denote the perturbed soliton for a real-valued and sufficiently small $v$,
we see that 
\begin{equation}
\label{nlkg-lin}
v_{tt} + \l v = N(v),
\end{equation}
where 
\begin{equation}
\label{estimates-lin}
\l := -\partial_x^2 + 1 - 4Q^3,
\qquad 
N(v) :=  6Q^2 v^2 + 4Q v^3 + v^4. 
\end{equation}
The spectral properties of the linearized operator $\l$ were described in details in \cite{CGNT07},
also in the case of general powers $u^p$ in \eqref{nlkg}.
In particular, it is known that $\l$ has continuous spectrum $[1, \infty)$ and exactly two isolated eigenvalues: 

\setlength{\leftmargini}{1.5em}
\begin{itemize}
\item[-] The first eigenvalue is $\lambda_0 = - \alpha (\alpha+2) = -\Omega^2$
($\Omega>0$) with corresponding even normalized eigenfunction
\begin{equation}
\label{even-efunction}
\rho(x) := c_0 (\sech(\alpha x))^{\frac{\alpha+1}{\alpha}}, \qquad \l \rho = -\Omega^2 \rho,
\end{equation} 
where the positive $c_0$ is chosen to satisfy $\|\rho\|_{L^2} = 1$; 

\item[-] The other eigenvalue is $\lambda_1 = 0$ with corresponding odd eigenfunction $Q'(x)$;
in particular, $\l$ has no gap eigenvalues (internal mode) in $(0,1)$.

\item[-] Moreover, the operator $\l$ has no resonance at the edge of the continuous spectrum,
i.e., there does not exist a bounded solution $\varphi$ to {$\l \varphi = \varphi$}.
Equivalently, the potential $V=  -4Q^3$ is `generic'.
\end{itemize}

To avoid the translation mode and having to modulate the soliton,
we restrict our analysis to even perturbations, and decompose $v$ as
\begin{equation}
\label{v-decomp}
v(t, x) = a(t) \rho(x) + \chi(t, x), \quad \lr{\chi,\rho}_{L_x^2} = 0,\quad \chi = P_c v,    
\end{equation}
where $P_c:L_x^2 \to L_x^2$ is the projection onto the continuous spectrum of $\l$.


\smallskip
\subsection{{Center-stable manifold} 
and some literature}
Before stating our main result we review the definition of the {center-stable manifold} around the soliton $Q$ for even solutions of \eqref{nlkg},
following Kowalczyk, Martel and Mu\~noz in \cite{KMM21} (see also Bates-Jones \cite{BJ89}),
as well as the main result in \cite{KMM21} concerning\footnote{The results in \cite{KMM21}
cover the more general case of nonlinearities $|u|^{p-1}u$ with $p>3$.} \eqref{nlkg}.
Rewrite the linear part of the equation \eqref{nlkg-lin} as the system
\begin{equation}
\label{linearized-system-v}
\left\{ 
\begin{array}{l}
\partial_t v_1 = v_2, 
\\
\partial_t v_2 = -\mathcal{L}v_1,
\end{array}
\right.
\end{equation}
where $\vec v := (v_1, v_2) = (v, v_t)$ and the operator $\mathcal{L}$ is defined in \eqref{estimates-lin}. 
The exponentially unstable and stable solutions of \eqref{linearized-system-v} 
are $\bm{v_+} (t,x) = e^{\Omega t} Y_+(x)$ and $\bm{v_-} (t,x) = e^{-\Omega t} Y_-(x)$, respectively, where 
\begin{equation}\label{Ypm}
Y_+ = \begin{pmatrix} \rho,  \Omega \rho \end{pmatrix} 
  \quad \text{and} \quad Y_- = \begin{pmatrix} \rho, -\Omega \rho \end{pmatrix}.
\end{equation}
We let
\begin{equation}\label{Zpm}
Z_+ = \begin{pmatrix} \rho,  \Omega^{-1} \rho \end{pmatrix}
\end{equation}
so that $
\lr{Y_-, Z_+} = 0$. 

For every $\delta_0>0$, define
\begin{equation}
\label{A-0}
\mathcal{A}(\delta_0) := \left\{  \bm{\gamma} \in H^1_x(\reals) \times L^2_x (\reals) : \quad \bm{\gamma} ~ \textrm{is even},~ \|\bm{\gamma}\|_{H^1_x \times L^2_x} < \delta_0 ~ \text{and}~ \lr{\bm{\gamma}, Z_+} = 0 \right\}.
\end{equation}

Then, the following stability result was obtained in \cite{BJ89,KMM21}.

\begin{theorem}[\cite{BJ89,KMM21}]
\label{thmSM}
There exist $C, \delta_0>0$ and a Lipschitz function $h: \mathcal{A}(\delta_0) \to \reals$ 
with $h(0)=0$ and $|h(\bm{\gamma})|\leq C\n{\bm{\gamma}}^{3/2}_{H_x^1 \times L^2_x}$ 
such that, denoting 
\begin{equation}
\label{center-manifold}
\mathcal{M}(\delta_0) := \{(Q, 0) + \bm{\gamma} + h(\bm{\gamma}) Y_+: ~ \bm{\gamma} \in \mathcal{A}(\delta_0)\}
\end{equation}
the following holds:

\begin{itemize}
\item[1.] If $(u_0, u_1) \in \mathcal{M}(\delta_0)$, then the solution of \eqref{nlkg} 
with  initial data $(u_0, u_1)$ is global and satisfies
\begin{equation}
\label{global-bound-KMM}
\big\|(u(t)-Q, \partial_t u(t))\big\|_{H_x^1 \times L^2_x}
  \leq C \big\|(u_0-Q, u_1)\big\|_{H_x^1 \times L^2_x} \quad \text{for all}~ t\geq 0.
\end{equation}

\smallskip
\item[2.] If a global even solution $(u, \partial_t u)(t)$ of \eqref{nlkg} satisfies 
\begin{equation*}
\big\|(u(t)-Q, \partial_t u(t))\big\|_{H_x^1 \times L^2_x} < \frac{1}{2}\delta_0 \quad \text{for all}~ t\geq 0,
\end{equation*}
then $(u, \partial_t u)(t) \in \mathcal{M}(\delta_0)$ for all $t\geq 0$.
\end{itemize}
\end{theorem}

The above classical theorem provides a {center-stable manifold} $\mathcal{M}=\mathcal{M}(\delta_0)$ such that data 
originating in $\mathcal{M}$ is global, 
and stays globally-in-time close to the soliton in the energy norm.
Moreover, any solution that is globally sufficiently close to $Q$, has to lie in $\mathcal{M}$. {Strictly speaking, construction of the manifold $\mathcal{M}$ in the original result by \cite{BJ89,KMM21} was done for the nonlinearity of type $|u|^{2\alpha} u$ with $\alpha>1$, but the proof carries forward for the nonlinearity of type $u^{2\alpha+1}$ as in \eqref{nlkg}.}

Note that, in view of the decomposition $u = Q+v = Q + a(t)\rho + \chi$, see \eqref{v-decomp}, 
if one chooses sufficiently small data $a(0), \dot{a}(0)$ 
and $(\chi(0,x), \chi_t(0,x)) \in H^1_x\times L^2_x$ small enough,
so that $(u(0),u_t(0)) \in \mathcal{M}(\delta_0)$,
then the global bound \eqref{global-bound-KMM} implies that, for every $t\geq 0$,
we have $\n{v(t)}_{H_x^1 \times L_x^2} \lesssim \delta_0$ and 
\begin{equation}
\label{global-bound-KMM-a-chi}
|a(t)| + |\dot a(t)| + \normm{\chi(t)}{x} + \normm{\partial_t \chi(t)}{x} \lesssim \delta_0.
\end{equation}

A natural question to ask is whether the above global solutions 
are also attracted to the soliton, besides remaining close to it for all times,
and what is their asymptotic behavior as $t \rightarrow \infty$.

For nonlinear equations with large enough power nonlinearities these types 
of asymptotic stability result have been achieved by several authors. 
{The literature is ample so we will only mention a few recent results,
and refer the readers to the cited works for more discussions and references.
For the Klein-Gordon equation with power nonlinearity $u^p$ with $p>5$
asymptotic stability (on the {center-stable manifold}) was proven by Krieger-Nakanishi-Schlag \cite{KNS12}
for data in the energy space.
See also \cite{NSbook,NS3dnr} and reference therein, for a similar problem
and more general results in higher dimensions.
See also the review 
by Kowalczyk-Martel-Mu\~noz \cite{KMMreview}. 
For the nonlinear Schr\"{o}dinger (NLS) equation 
Krieger-Schlag \cite{KriSchNLS} treated the case of a $|u|^{2\alpha}u$ nonlinearity with $\alpha>2$ in $1$d
and Cuccagna-Pelinovsky \cite{CP14} the cubic case.
See also the seminal work of Schlag \cite{Sch09} for the cubic NLS in $3$d,
the recent works and survey by Cuccagna-Maeda \cite{CM21B,CM21}
and Maeda-Yamazaki \cite{MaedaYamazaki} and references therein. 
}

For low power nonlinearities instead (typically the {mass sub-critical} ones)
one cannot expect similar asymptotic stability result
in the energy space $H^1_x(\R) \times L^2_x(\R)$, essentially because Strichartz estimates
are not very effective in controlling the nonlinearity.
There are then two natural different and complementary ways to look at this question:
(a) studying the asymptotics for energy solutions locally in the space variable $x$,
or 
(b) studying asymptotics on the whole real line in a smaller space, e.g., 
a weighted Sobolev space. 
The first approach is usually based on delicate virial type estimates, and has been 
very successfully employed in several problems;
see for example the already mentioned \cite{KMM21}, the recent work of Li-L\"uhrmann \cite{LiLu22} on the soliton of quadratic KG, Cuccagna-Maeda-Scrobogna \cite{CMS22}
on cubic KG, and the works on kink solutions for relativistic scalar fields \cite{KowMarMun,KowMarMunVDB,CM22}.
The second approach has instead been pursued mostly in the context of 
kinks (non-localized solitons) which naturally lead to nonlinear PDEs with potentials and 
low power non-localized nonlinearities;
we refer the reader to Delort-Masmoudi \cite{DMKink}, Germain and the second author \cite{GP20} and Zhang
\cite{KGVSim},
L\"uhrmann-Schlag \cite{LSch}, Lindblad-L\"uhrmann-Soffer \cite{LLS20},
and Lindblad-L\"uhrmann-Schlag-Soffer \cite{LLSS}.

\smallskip
Concerning the problem we are interested in,
\cite{KMM21} showed the following asymptotic stability result

\begin{theorem}[{\cite[Theorem 1]{KMM21}}]\label{thmKMM}
Consider the Klein-Gordon equation
\begin{align}\label{nlkgalpha}
\partial_t^2 u - \partial_x^2 u + u = |u|^{2\alpha}u, \qquad \alpha > 1,
\end{align}
and denote its static soliton by
\begin{equation}\label{solitonalpha}
Q(x) := (\alpha+1)^{\frac{1}{2 \alpha}} \sech^{1/\alpha} (\alpha x).
\end{equation}
Then, there exists a constant $\delta >0$ such that if a global even solution $(u,u_t)$ of \eqref{nlkgalpha}
satisfies
\begin{align}\label{KMMas}
\big\|(u(t)-Q, \partial_t u(t))\big\|_{H_x^1(\R) \times L^2_x(\R)} < \delta \quad \text{for all}~ t\geq 0, 
\end{align}
then, for any bounded interval $I$, 
\begin{align}\label{KMMconc}
\lim_{t\rightarrow \infty} \big\|(u(t)-Q, \partial_t u(t))\big\|_{H_x^1(I) \times L^2_x(I)} = 0.
\end{align}
Moreover,
\begin{align}\label{KMMconcint}
\int_0^\infty \big\|(u(t)-Q, \partial_t u(t))\big\|_{H_x^1(I) \times L^2_x(I)}^2 \, dt < \infty.
\end{align}
\end{theorem}

Note that a global solution satisfying \eqref{KMMas} for $\delta$ small enough 
is on the {center-stable manifold} and therefore stays close to the soliton in view of Theorem \ref{thmSM}.
Theorem \ref{thmKMM} then gives that these solutions actually converge as $t \rightarrow \infty$ to 
$Q$ locally-in-space.
Our main theorem will give a complementary asymptotic stability result;
we will assume more localization on the (well-prepared) 
initial data and prove convergence to $Q$ on the whole real-line
with optimal time-decay rates, and local-decay estimates as well.
In particular, we will show that, see \eqref{pointwise-decay-main-theorem},
\begin{align}\label{concint01}
& \int_0^\infty \big\| 
  u(t)-Q 
  \big\|_{L^\infty_x(\R)}^{2+} \, dt < \infty,
\\
\label{concint02}
& \int_0^\infty \big\|\langle x \rangle^{-2} (
  u(t)-Q 
  )
  \big\|_{L^\infty_x(\R)}^{1+} \, dt < \infty.
\end{align}

\smallskip
\subsection{Main result}
Recall that we write $u(t,x) = Q(x) + v(t,x) = Q(x) + a(t)\rho(x) + \chi(t,x)$,
where $\rho$ is given in \eqref{even-efunction},
$Y_\pm := (\rho, \pm \Omega \rho)$ and $\mathcal{M}$ is defined in \eqref{center-manifold} with \eqref{Zpm}.
We will consider well-prepared initial data in a subset of $\mathcal{M}$ as we now describe.
Let $\bm{\gamma}(x) = (\gamma_1(x), \gamma_2(x))$ 
be given in the form (see \eqref{A-0})
\begin{equation*}
\bm{\gamma}(x) = b Y_-(x) + \bm{\zeta}(x), \quad \bm{\zeta} = (\zeta_1, \zeta_2),
\end{equation*}
where $b$ is a scalar constant, and $\bm{\zeta}$ is even in $x$, with $\lr{\rho, \zeta_j} = 0$ for $j=1,2$. 
In other words, $\bm{\gamma}$ has components
\begin{equation}
\label{gamma-12-form}
\begin{aligned}
\gamma_1(x) = b \rho(x) + \zeta_1(x), \quad \gamma_2(x) = -\Omega b \rho(x) + \zeta_2(x).
\end{aligned}
\end{equation}
We define (see \eqref{estimates-lin})
\begin{align}\label{defH}
H := P_c (-\partial_x^2 - 4Q^3), \quad H+1 = P_c \l, 
\end{align}
where $P_c$ denotes the projection onto the continuous spectrum,
and require that
\begin{equation}
\label{initial-data-gamma}
|b| + \n{(\sqrt{H+1}~ \zeta_1, \zeta_2)}_{H_x^2} + \n{\lr{x}(\sqrt{H+1}~ \zeta_1, \zeta_2)}_{L_x^2} \leq c\eps_0
\end{equation}
for sufficiently small $\eps_0>0$, with $c$ a small absolute constant ($c=1/10$, say). 
Then, $\bm{\gamma}$ belongs to a strict subset of $\mathcal{A}(\eps_0)$.  
We then define the following subset of $\mathcal{M}(\eps_0)$:
\begin{equation}
\label{center-manifold-subset}
\mathcal{M}_0(\eps_0) := \{(Q,0) + \bm{\gamma} + h(\bm{\gamma}) Y_+: \quad 
  \bm{\gamma}~ \text{is even and satisfies \eqref{gamma-12-form}--\eqref{initial-data-gamma}} \} 
  \subseteq \mathcal{M}(\eps_0),
\end{equation}
where $h$ is the Lipschitz function from Theorem \ref{thmSM}.
We are now ready to state our main result:

\begin{theorem}\label{main-theorem}
There exists sufficiently small $\eps_0>0$ such that, for every $(u_0, u_1) \in \mathcal{M}_0(\eps_0)$,
the equation \eqref{nlkg} with initial data $(u, \partial_t u)(0) = (u_0, u_1)$ 
admits a unique global solution of the form 
\begin{equation}
\label{main-theorem-solution}
u(t,x) = Q(x) + a(t) \rho(x) + \chi (t,x), \quad \lr{\rho, \chi}_{L_x^2} = 0,
\end{equation}
that satisfies, for all $t\geq 0$, \eqref{global-bound-KMM} and the 
following global-in-time decay bounds:





\begin{align}
\label{pointwise-decay-main-theorem}
|a(t)| \lesssim \eps_0 \lr{t}^{-2}, \qquad \n{\chi}_{L_x^\infty} \lesssim \eps_0 \lr{t}^{-1/2},
\qquad 
\n{ \lr{x}^{-2} \chi}_{L_x^\infty} \lesssim \eps_0 \lr{t}^{-1}.
\end{align}
Moreover, $\chi$ behaves like a free solution of the Klein-Gordon equation as $t \rightarrow \infty$.

\end{theorem}

\medskip
Here are a few remarks about our result:


\begin{remark}[About the initial data]\label{remdata2}
The choice of the initial data can also be interpreted as follows.
We can express $(u_0, u_1)$ as 
\begin{align*}
& u_0(x) = Q(x) + a_0 \rho(x) + \chi_0(x), 
\\
& u_1(x) = a_1 \rho(x) + \chi_1(x),
\end{align*}
with $\lr{\chi_j,\rho}_{L^2_x}=0$, $j=1,2$,
$(a_0,a_1) = (a(0),\dot{a}(0))$, $(\chi_0,\chi_1) = (\chi(0),\partial_t\chi(0))$.
Then, 
\begin{equation*}
(u_0, u_1) = (Q, 0) + \frac{1}{2} \Big( a_0 - \frac{1}{\Omega} a_1 \Big) Y_- 
  + \frac{1}{2} \Big( a_0 + \frac{1}{\Omega} a_1 \Big) Y_+ 
  + (\chi_0, \chi_1),
\end{equation*}
which, compared with \eqref{gamma-12-form} would give
$b = (1/2)(a_0 - \Omega^{-1} a_1)$, $(\zeta_1,\zeta_2) = (\chi_0,\chi_1)$.
Then $(u_0, u_1) \in \mathcal{M}_0(\eps_0)$ 
if we impose the stability condition
\begin{equation*}
\frac{1}{2} \Big( a_0 + \frac{1}{\Omega} a_1 \Big) 
= h\Big( \frac{1}{2} \Big( a_0 - \frac{1}{\Omega} a_1 \Big) Y_- + (\chi_0, \chi_1) \Big),
\end{equation*}
that $\chi_0,\chi_1$ are even, and 
\begin{equation}
\label{initial-data-chi}
|a_0| + |a_1| + \n{(\sqrt{H+1}~ \chi_0, \chi_1)}_{H_x^2} + \n{\lr{x}(\sqrt{H+1}~ \chi_0, \chi_1)}_{L_x^2} 
  \leq c \eps_0.
\end{equation}
for some absolute $c>0$ small enough.
\end{remark}

\begin{remark}[Linearized operator]\label{remlinop}
The potential appearing in the linearized operator \eqref{estimates-lin} is 
$V = -4Q^3 = -10 \cosh^{-2}(3x/2)$.
The corresponding Schr\"odinger operator is, up to a conjugation, given by
$H = -\partial_x^2 + c \cosh^{-2}(x)$ with $c= -40/9$.
In particular, $V$ is not in the class of P\"oschel-Teller potentials 
$V_\ell(x) = - \ell(\ell+1)\cosh^{-2}(x)$, $\ell=0,1,2,\dots$ 
that can be conjugated to the flat $V=0$ operator;
see the discussion in \cite{LSch} and references therein.
\end{remark}

\begin{remark}[More general powers 1]\label{remp}
Consider the pure power Klein-Gordon equation
\begin{align}\label{nlkgp}
\partial_t^2 u - \partial_x^2 u + u = |u|^{2\alpha} u,
\end{align}
and its soliton
\begin{equation}\label{solitonp}
Q_\alpha(x) = (\alpha+1)^{\frac{1}{2 \alpha}} \sech^{1/\alpha} (\alpha x), 
\end{equation}
Recall that Theorem \ref{thmKMM} applies to all powers $\alpha > 1$.
In the borderline cubic case, $\alpha=1$, the linearized operator
has an even resonance and the results of \cite{KMM21} 
do not apply; 
{in this case the only known result so far is due to \cite{LSch23}, 
where co-dimension one stability and pointwise decay for the radiation are proven
for localized data (as in this paper) of size $\epsilon$, up to times of the order $O\big(\exp(c\epsilon^{-1/4})\big)$.}


In the quadratic case, $\alpha = 1/2$, the linearized operator has an internal mode (a gap eigenvalue
in $(0,1)$).
This is a more difficult case since stability does not hold at the linearized level
because of the presence of periodic-in-time localized bound states.
Nevertheless, a stability result like the one in Theorem \ref{thmKMM} was recently proven
by Li and L\"uhrmann \cite{LiLu22}.
See this latter, \cite{KowMarMun}, and the work of L\'eger and the second author \cite{LP22}
for more references about internal mode dynamics.


\end{remark}

\begin{remark}[More general powers 2]
{
Our main result in Theorem \ref{main-theorem} corresponds to \eqref{nlkgp} with the nonlinearity $|u|^{2\alpha} u$ 
replaced by $u^{2\alpha+1}$ and the case $2\alpha=3$. The result can be easily extended 
to any $\alpha$ such that $2\alpha$ is an integer larger $3$.}
{For fractional powers, on the other hand, the situation is a little different
and we cannot use directly our approach. 
In fact, applying the (distorted) Fourier transform to the nonlinear terms will not 
turn them into convolution-type expressions where the (linear) oscillations of every
copy of the solution can be easily factored and combined into a nonlinear phase; 
see \eqref{wg-evolutionintro}-\eqref{wg-duhamelintro} and \eqref{quartic-integral}-\eqref{spectral-distribution}.}
{For similar reasons, the approach cannot be used as is for nonlinearities involving absolute values, 
such as $|u|^{2\alpha} u$ {{with non-integer $\alpha$.}}
However, it is natural to conjecture that the same result would hold {for any $\alpha>1$
consistently with the predictions of \cite{BiChSz11} (see Remark \ref{Bizon} for more about this) 
and the result and techniques in \cite{NauWed,CGV}.}}
\end{remark}

\begin{remark}[Pointwise and integrated decay on the real line]
Our main Theorem \ref{main-theorem} complements known results in two ways.
First, it complements \cite{KMM21} by showing that asymptotic stability holds on the whole real line.
Second, it complements results for larger powers, such as \cite{KNS12}, 
by obtaining (also for lower power nonlinearities)
exact rates of convergence to the soliton for a large class of solutions trapped by it. 
As already mentioned, we need to require stronger assumptions on the data, the most relevant of which
is the smallness of the weighted norm in \eqref{initial-data-chi}.

Also note that the first and second estimate in \eqref{pointwise-decay-main-theorem} 
immediately imply the integrated decay \eqref{concint01}, while the first and third 
give \eqref{concint02}.\footnote{Similar estimates with $\partial_t u$ replacing $u-Q$ can also be obtained, 
but one would need to change slightly the assumptions on the data in terms of derivatives;
for simplicity, we do not pursue this here.}
\end{remark}

\begin{remark}[About local decay and rates of convergence]
\label{Bizon}
Algebraic rates of decay towards the soliton 
on the {center-stable manifold} had been predicted numerically already in \cite{BiChSz11}.
In particular, Fig. 2 in \cite{BiChSz11} depicts the expected decay (or growth, for larger data)
in the quartic case \eqref{nlkg}.
Note that the data used for these numerical results is a Gaussian. 

The predictions of \cite{BiChSz11} actually give a local-in-space rate of decay of $t^{-3/2}$,
as opposed to the rate of $t^{-1}$ which we obtain from \eqref{pointwise-decay-main-theorem}.
The discrepancy here is just due to the fact that we only require smallness of $xu(0) \in L^2$.
For such data, even at a the level of the linear operator $e^{it\sqrt{H+1}}$, 
one cannot obtain more than $t^{-1}$ local decay.
Assuming the stronger localization $xu(0) \in L^1$, the linear decay rate is indeed $t^{-3/2}$,
and one can expect to obtain this same rate of convergence to $Q$ 
also for our nonlinear solutions. 
\end{remark}

\begin{remark}[{About the parity assumption}]\label{Remarkeven}
{Our main result in Theorem \ref{main-theorem} applies to even solutions.
Since the Klein-Gordon equation is invariant under translations and Lorentz transformations,
in the general case of data without parity, 
one needs to consider the whole family of soliton solutions 
\begin{align}\label{Remeven1}
Q(\gamma(x-\beta t-x_0)), \qquad \gamma := \frac{1}{\sqrt{1-\beta^2}}, \quad |\beta|<1.
\end{align}
The proof of an asymptotic stability result in this case would be substantially more involved;
it would require the linearization around a moving soliton of the form $Q(\gamma(t)(x-q(t)))$ and

(1) A precise study of a more complicated linear operator in matrix form 
(this will have additional off-diagonal first order terms compared to the operator appearing in \eqref{linearized-system-v});

(2) An extension of the analysis of the distorted Fourier transform and of the nonlinear spectral distribution
adapted to this new linearized operator;

(3) A modulation analysis with refined estimates for the time-dependent parameters of the moving soliton.

Our expectation is that all these issues could be tackled using a combination of 
(a) classical tools, for example those concerning the set-up of the modulation analysis
and the leveraging of local decay, see \cite{KriSchNLS}
and (b) our techniques based on the distorted Fourier transform 
which can give a more precise description of the linear flow and of the nonlinear oscillations, 
and sharper estimates on the evolution of the modulation parameters.
}
{We refer to the recent work \cite{ColGer23} where
some of these difficulties have been addressed for some Schr\"odinger type equations.}



{The expected asymptotic stability result for general initial data sufficiently close to a soliton \eqref{Remeven1}
would then look as follows:
for well prepared initial data (on the center-stable manifold) 
of the form $u_0(x) = Q(\gamma_0(x-q_0)) + v_0(x)$, with $v_0$ small in a weighted Sobolev space,
the solution will satisfy
\begin{equation*}
u(t,x) = Q\big(\gamma(t) (x - q(t))\big) + v(t,x)
\end{equation*}
where the modulation parameters $\gamma(t),q(t)$ converge as $t\rightarrow \infty$,
and the radiation $v$ disperses at the linear rate, ${\| v(t) \|}_{L^\infty_x} \lesssim \lr{t}^{-1/2}$.
}


\end{remark}


\smallskip
\subsection{Strategy of the proof}
We begin by following the standard strategy
of substituting the decomposition \eqref{v-decomp} into \eqref{nlkg-lin}-\eqref{estimates-lin}, 
and separating the discrete and continuous modes as follows:
\begin{equation}
\label{system-eq}
\left\{ \begin{array}{l}
a_{tt} - \Omega^2 a = \lr{N(v), \rho}_{L_x^2} := F[v], 
\\
\chi_{tt} + \l \chi = P_c N(v).
\end{array} \right.
\end{equation}
Following the Duhamel's principle, \eqref{system-eq} is equivalent to the system of integral equations
\begin{equation}
\label{integral-eq-0}
\left\{ \begin{array}{l}
a(t) = \cosh(\Omega t) a(0) + \dfrac{1}{\Omega} \sinh(\Omega t) \dot a (0)
+ \dfrac{1}{\Omega} \displaystyle{\int_0^t} \sinh(\Omega(t-s)) F[v](s) ds, 
\\
\chi(t) = \cos (\sqrt{H+1} t) \chi(0) + \dfrac{\sin(\sqrt{H+1} \, t)}{\sqrt{H+1}} \dot \chi (0) + 
  \displaystyle{\int_0^t} \dfrac{\sin(\sqrt{H+1} (t-s))}{\sqrt{H+1}} P_c N (v(s)) ds.
\end{array} \right.
\end{equation}
Imposing the \textit{stability condition}
\begin{equation}
\label{stability-cond}
\dot a (0) = -\Omega a (0) - \int_0^\infty e^{-\Omega s} F[v](s) ds, 
\end{equation}
which ensures that the contribution of the exponentially 
increasing term $e^{\Omega t}$ of $a(t)$ vanishes as $t \to \infty$, 
\eqref{integral-eq-0} reduces to
\begin{equation}
\label{integral-eq}
\left\{ \begin{array}{l}
a(t) = e^{-\Omega t} \left[ a (0) + \dfrac{1}{2\Omega} \displaystyle{\int_0^\infty} e^{-\Omega s} F[v](s) ds \right] - 
  \dfrac{1}{2\Omega} \displaystyle{\int_0^\infty} e^{-\Omega |t-s|} F[v](s) ds, 
\\
\chi(t) = \cos (\sqrt{H+1} t) \chi(0) + \dfrac{\sin(\sqrt{H+1} \, t)}{\sqrt{H+1}} \dot \chi (0) + 
  \displaystyle{\int_0^t} \dfrac{\sin(\sqrt{H+1} (t-s))}{\sqrt{H+1}} P_c N(v(s)) ds.
\end{array} 
\right.
\end{equation}
Note that \eqref{stability-cond} holds for solutions on the {center-stable manifold},
in view of \eqref{global-bound-KMM}.

To rewrite the evolution equation for $\chi(t,x)$ in \eqref{system-eq} 
in a convenient way for our analysis, we introduce the profile function $g(t,x)$ as
\begin{equation}
\label{g-profile}
g(t, x) := e^{it \sqrt{H+1}}  ( \partial_t - i \sqrt{H+1} ) \chi(t, x).
\end{equation}
Then the real-valued function $\chi$ is given in terms of $g$ and $\bar g$ by
\begin{equation}
\label{chi-in-terms-of-g}
\chi(t, x) = \frac{1}{2i \lrH} (e^{it \lrH} \bar g(t,x) -  e^{-it \lrH}  g(t, x)).
\end{equation}
We then want to look at Duhamel's formula in terms of $g$ on the distorted Fourier side.
We refer the reader to Section \ref{secdFT} for the definition of the distorted Fourier 
Transform (dFT) and its properties.
Applying the distorted Fourier transform to \eqref{g-profile}
gives $\widetilde g(t, k) = e^{it \lr{k}} ( \partial_t - i \lr{k} ) \widetilde \chi(t,k)$,
and the second equation in \eqref{integral-eq} becomes 
\begin{equation}
\label{wg-evolutionintro}
\partial_t \wg (t, k) = e^{it\lr{k}} \widetilde{\mathcal{F}} P_c N(Q, v),
  \qquad N(Q, v) := 6Q^2 v^2 + 4 Q v^3 + v^4.
\end{equation}
Integrating \eqref{wg-evolutionintro} over time gives
\begin{equation}
\label{wg-duhamelintro}
\wg (t,k) = \wg (0,k) + \int_0^t e^{is\lr{k}} \distF P_c N (s, k) ds.
\end{equation}
Equation \eqref{wg-duhamelintro} is the starting point for our analysis of the 
projection on the continuous spectrum of the solution. 

To obtain the desired nonlinear estimates that will imply our main result,
we introduce the following norm 
\begin{equation}
\label{main-normintro}
\|v\|_{X} := \|\lr{t}^2 a(t)\|_{L^\infty_t (0, \infty)} + 
\|\prtl_k \wg \|_{L^\infty_t ((0, \infty): L_k^2)} + \|\lr{k}^2 \wg \|_{L^\infty_t ((0, \infty): L_k^2)},
\end{equation}
The first norm above is used to control the evolution of the solution in the direction of the 
even eigenfunction, while the other two control the evolution of the continuous spectral component.
The proof of our result will follow from a priori estimates for this norm:

\begin{proposition}\label{propmain}
For $\n{v}_X$ sufficiently small we have
\begin{equation*}
\| v \|_X \lesssim  |a(0)| + \normm{\partial_k \wg (0)}{k} + \normm{\lr{k}^2 \wg (0)}{k} + \|v\|_X^2.
\end{equation*}
\end{proposition}

Proposition \ref{propmain} follows from the three Propositions \ref{a(t)-norm-estimates},
\ref{proposition-partial-k-L2-norm} and \ref{proposition-k2-wg-norm-bound}.
With Proposition \ref{propmain} we can prove the main theorem.

\smallskip
\subsection{Proof of Theorem \ref{main-theorem}}
First, given initial data $(u_0,u_1) \in \mathcal{M}_0(\epsilon_0)$ 
with \eqref{gamma-12-form}-\eqref{initial-data-gamma},
Theorem \ref{thmSM} gives a global solution trapped by the soliton for all $t\geq 0$
provided $\epsilon_0$ is chosen small enough depending on $\delta_0$.

The smallness condition 
\eqref{initial-data-chi} on the initial data
with (recall \eqref{g-profile})
$(-\sqrt{H+1}\chi_0,\chi_1) = (\Im(g(0)), \Re(g(0)))$
and the estimates 
\eqref{derivative-in-fourier-bound} and \eqref{k-s-h-tilde-bound}, imply that
\begin{align}
|a(0)| + \normm{\partial_k \wg (0)}{k} + \normm{\lr{k}^2 \wg (0)}{k} \leq C \epsilon_0
\end{align}
for some absolute constant $C>0$.
Then, Proposition \ref{propmain} gives, for a possibly different $C>0$,
$\| v \|_X \leq C\epsilon_0 + C \|v\|_X^2$.

Therefore, we see that 
the map defined by
the right-hand side of the first equation in \eqref{integral-eq} 
and by the right-hand side of \eqref{wg-duhamelintro} 
takes a ball of radius $2C\eps_0$, in $X$, which we denote by $B_{2C \eps_0}^X$, 
into itself for some absolute $C>0$. 
Moreover, repeating the same arguments one can verify
that the map is a contraction, 
hence, there is a unique fixed point 
in $B_{2C \eps_0}^X$.

Finally, the bound $\| v \|_X \lesssim \epsilon_0$ 
implies directly the first bound in \eqref{pointwise-decay-main-theorem}.
The other two estimates in \eqref{pointwise-decay-main-theorem} 
are obtained as a direct consequence 
of the pointwise and local decay estimates in Lemmas \ref{lemma-dispersive-estimates}
and \ref{local-l-infty-decay-lemma}. $\hfill \Box$

\smallskip
\subsection{Notations}
\label{subsection-notations}
We let $\lr{x} := (1+x^2)^{\frac{1}{2}}$. 
We write $a \lesssim b$ if there is a uniform constant $C$ such that $a \leq Cb$. 
Indicator functions in variable $y$ will be denoted as $\mathbf{1}_I(y)$
meaning the function is equal to $1$ whenever $y\in I$ and vanishes otherwise. 
To distinguish between the function $h$ and its complex conjugate, in a number of occasions, 
we use $h^{(+)}(x) := h(x)$ to denote the function itself and $h^{(-)}(x) := \overline{h}(x)$ 
to denote the complex conjugate of $h$.

\smallskip
\subsection{Acknowledgments}
A. K. was supported in part by NSERC grant RGPIN 2018-04536,
and thanks the Fields Institute and McMaster University for its support.

\noindent
F.P. was supported in part by a start-up grant from the University of Toronto, 
and NSERC grant RGPIN-2018-06487.


\medskip
\section{Linear Theory}\label{secdFT}

In this section, we consider 
the second order differential operator $H := -\partial_x^2 + V$, 
where $V(x)$ is a potential satisfying the following assumption 

\begin{assumption}
\label{V-assumption}
The potential $V$ is a smooth function exponentially decaying at infinity,
\begin{equation}
\label{potential-decay}
\lim_{|x| \to \infty} e^{\alpha |x|} V(x) = 0 \quad \text{for some } \alpha>0,
\end{equation}
and $H = -\prtl_{xx} + V$ has no point spectrum, and no zero energy resonance, 
i.e. there is no bounded nontrivial solution in the kernel of $H$
or, equivalently, $V$ is `generic' in the sense of Definition \ref{defGen}.

\end{assumption}

In our application to the equation \eqref{nlkg-lin}, the role of the operator $H$ 
is played by $P_c \mathcal{L}$.
For an operator $H$ as above we present the basic facts about Jost functions, 
reflection and transmission coefficients, 
and the distorted Fourier transform.
This theory is based on classical results of Deift-Trubowitz \cite{DT79} and Faddeev \cite{F59,F64}.
See also the book of Yafaev \cite{Y10}.

\begin{remark}
As it was proven in \cite{CGNT07},
the potential $-4Q^3$ given in \eqref{estimates-lin} is generic. 
\end{remark}

\subsection{Jost functions} 
We define $f_+(x, k)$ and $f_-(x, k)$ as solutions to the differential equation
\begin{equation}
\label{ode-Jost}
H f_\pm (x, k) = k^2 f_\pm (x, k)
\end{equation}
which satisfy the asymptotic boundary conditions
\begin{equation}
\label{f-pm-asymptotics}
\lim_{x \to \infty} |e^{-ikx} f_+(x, k) - 1| = 0 \quad \text{and} \quad \lim_{x \to -\infty} |e^{ikx} f_-(x, k) - 1| = 0
\end{equation}
We further define $m_\pm (x, k) := e^{\mp ikx} f_\pm (x, k)$. For fixed $x$, the functions $m_\pm (x,k)$ are analytic in $\Im k >0$ and continuous in $\Im k \geq 0$. The following lemma provides with the primary estimates on  $m_\pm$, which will be used later.
\begin{lemma}
\label{bounds-m-pm}
For every $s = 0,1,2$ and every positive integer $N$, we have
\begin{equation}
\label{estimates-partial_k^s_m}
\begin{aligned}
&|\partial_k^s (m_\pm (x, k) - 1)| \lesssim \lr{x}^{-12} \lr{k}^{-1}, \quad \qquad   \pm x \geq -1, \\
&|\partial_k^s (m_\pm (x, k) - 1)| \lesssim  \lr{x}^{s+1} \lr{k}^{-1}, \quad \qquad \pm x \leq 1.
\end{aligned}
\end{equation}
Moreover, in the presence of the $x$ derivatives, we have 
\begin{equation}
\label{estimates-partial_k_and_x_m}
\begin{aligned}
&|\partial_x^N \partial_k^s m_\pm (x, k)| \lesssim \lr{x}^{-12} \lr{k}^{-1+N}, \quad \qquad   \pm x \geq -1, \\
&|\partial_x^N \partial_k^s m_\pm (x, k)| \lesssim \lr{x}^{s+1} \lr{k}^{-1+N}, \qquad \qquad \pm x \leq 1.
\end{aligned}
\end{equation}
\end{lemma}
\begin{proof}
The estimates follow by iterating the Volterra integral equations for $m_\pm$ given by
\begin{equation*}
m_\pm (x, k) = 1 \pm \int_{x}^{\pm \infty} D_k(\pm (y-x)) V(y) m_\pm(y,k) dy, \quad \text{where  } D_k(y) = \frac{e^{2iky}-1}{2ik}
\end{equation*}
as in Deift-Trubowitz \cite{DT79} or Appendix A in \cite{GPR18} under the assumption \eqref{potential-decay} on $V$.
\end{proof}

\begin{remark}
Results similar to those in Lemma \ref{bounds-m-pm} can be derived for any $s \geq 0$, 
and in such case some of the bounds should be modified, 
see for example \cite{CP21, GP20}. However, we will be working only with up to second derivatives 
of $m_\pm$, and so are the choices of $s$.
\end{remark}

The Jost functions $f_\pm$ are related through so-called transmission and reflection coefficients, $T(k)$ and $R_\pm(k)$, respectively \cite{DT79}. That is, 
\begin{equation}
\label{jost-functions-via-scat-par}
\begin{aligned}
T(k) f_+(x,k) & = R_-(k) f_-(x,k) + f_-(x, -k),
\\
T(k) f_-(x,k) & = R_+(k) f_+(x,k) + f_+(x, -k),
\end{aligned}
\end{equation}
and using the boundary conditions \eqref{f-pm-asymptotics}, we obtain 
\begin{equation}
\begin{aligned}
T(k) f_+(x,k) & \sim R_-(k) e^{-ikx} + e^{ikx} \quad \text{as } x\to -\infty
\\
T(k) f_-(x,k) & \sim R_+(k) e^{ikx} + e^{-ikx} \quad \text{as } x\to \infty.
\end{aligned}
\end{equation}
The coefficients $T(k)$ and $R_\pm(k)$ can be written explicitly through 
\begin{equation}
\label{T-R-coefficients-explicit}
\begin{aligned}
\frac{1}{T(k)} & = 1 - \frac{1}{2ik} \int V(x) m_\pm(x,k) dx, 
\\
\frac{R_\pm(k)}{T(k)} & = \frac{1}{2ik} \int e^{\mp 2ikx} V(x) m_\mp(x,k) dx,
\end{aligned}
\end{equation}
and satisfy the following relations at every $k$:
\begin{equation*}
\begin{aligned}
& T(-k) = \overline{T(k)}, \quad R_\pm (-k) = \overline{R_\pm(k)}, 
\\
& |T(k)|^2 + |R_\pm(k)|^2 = 1, \quad T(k) \overline{R_-(k)} + \overline{T(k)} R_+(k) = 0.
\end{aligned}
\end{equation*}
One distinguishes
two possible asymptotics of $T(k)$ for $k \approx 0$, 
depending on the value of the integral $\int_{-\infty}^\infty V(x) m_+(x,0) dx$,
which is zero whenever the potential $V$ is \textit{generic}. 
Here are some equivalent definitions of generic potentials. 

\begin{definition}\label{defGen}
The potential $V$ is generic if one of the following, equivalent, conditions hold: 
\begin{itemize}
\item $\int_{-\infty}^\infty V(x) m_+(x,0) dx = 0$;
\item $T(0) = 0$ and $R_\pm (0) = -1$;
\item $H = -\partial_x^2 + V$ has no zero resonance. 
\end{itemize}
\end{definition}

For generic potentials, $T(k)$ and $R_\pm (k)$ have the following asymptotics around $k \approx 0$. 

\begin{lemma}
\label{lemma-T-R-properties}
For $V$ satisfying Assumption \ref{V-assumption},
the coefficients $T(k)$ and $R_\pm(k)$ have the following properties:
\begin{itemize}
\item as $k\to 0$, 
\begin{equation}
\label{T-behaviour-around-zero}
\begin{aligned}
T(k) & = \alpha k + \calO (k^2) \quad \text{for some } \alpha \neq 0,
\\
1 + R_\pm(k) & = \alpha_\pm k + o(k).
\end{aligned}
\end{equation}
In particular, $T(0) = 0$ and $R_\pm (0) = -1$.
\item For every $k\in \mathbb{R}$ we have
\begin{equation}
\label{uniform-estimate-T-R}
\begin{aligned}
|T(k)| + |R(k)| &\lesssim 1,\\
|\partial_k T(k)| + |\partial_k R_\pm(k)|& \lesssim \lr{k}^{-1}.
\end{aligned}
\end{equation}
\item For every $k \in \reals$ we have 
\begin{equation}
\label{bound-on-T-over-k}
\left| \frac{T(k)}{k} \right| + \lr{k} \left| \partial_k \left(\frac{T(k)}{k}\right) \right| \lesssim \frac{1}{\lr{k}},
\end{equation}
and 
\begin{equation}
\label{bound-on-R-over-k}
\left| \frac{1+R_\pm(k)}{k} \right| + \left| \partial_k \left(\frac{1+R_\pm(k)}{k}\right) \right| \lesssim \frac{1}{\lr{k}}.
\end{equation}
\end{itemize}
\end{lemma}

\begin{proof}
The detailed proof of the estimates \eqref{T-behaviour-around-zero}--\eqref{uniform-estimate-T-R} 
can be found in Deift-Trubowitz \cite{DT79}, so we only prove the 
estimate \eqref{bound-on-T-over-k} and the proof of \eqref{bound-on-R-over-k} is similar.  
It is worth pointing out that Assumption \ref{V-assumption} on $V$ 
is stronger than needed for this lemma to be true, and it suffices to have $\lr{x}^2 V \in L_x^1$.

The asymptotic of $T(k)$ around $k \approx 0$ imply that there exists a positive constant $k^*$ such that 
\begin{equation*}
\left| \frac{T(k)}{k} \right| \lesssim \frac{1}{\lr{k}} \quad \text{for every   } |k| \leq k^*.
\end{equation*}
On the other hand, for $|k|>k^*$, we use uniform boundedness of $|T(k)|$ to get
\begin{equation*}
\left| \frac{T(k)}{k} \right| \lesssim \frac{|T(k)|}{\lr{k}}
  \lesssim \frac{1}{\lr{k}}.
\end{equation*}
Combining these estimates together we obtain the first part of \eqref{bound-on-T-over-k}. To get the bound on the derivative part, we note that
\begin{equation*}
\prtl_k \left( \frac{T(k)}{k} \right) = \left( \frac{T(k)}{k}\right)^2 \left(\frac{1}{2i} \int V(x) \partial_k m_\pm (x,k) dx - 1 \right),
\end{equation*}
which leads to a desired estimate in the view of \eqref{estimates-partial_k^s_m} and assumptions on $V$.
\end{proof}

\subsection{Pseudo-differential bounds}
We provide some $L^2 \to L^2$ bounds for pseudodifferential
operators whose symbols include $m_\pm(x,k)$ and their derivatives.
Under milder assumptions than ours, which in particular do not require any regularity for $V$,
these types of bounds are given in \cite{CP21}.

\begin{lemma}
\label{lemma-pseudo-bound-1}
For every function $a\in L_k^\infty$ we have
\begin{equation}
\label{pseudo-bound-1}
\normm{\int \mathbf{1}_{[-1,\infty)} e^{ikx} a(k) (m_+(x,k) - 1) h(x) dx}{k} \lesssim \normm{\lr{x}^{-11} h}{x},
\end{equation}
and 
\begin{equation}
\label{pseudo-bound-2}
\normm{\int \mathbf{1}_{[-1,\infty)} e^{ikx} a(k) \partial_k m_+(x,k) h(x) dx}{k} \lesssim \normm{\lr{x}^{-10} h}{x}.
\end{equation}
Similar estimates hold for $\mathbf{1}_{(-\infty, 1]}(m_-(x,k) - 1)$ 
and $\mathbf{1}_{(-\infty, 1]}\partial_k m_-(x,k)$, respectively.
\end{lemma}

\begin{proof}
Using the estimates \eqref{estimates-partial_k^s_m}, 
we have that the left-hand side of \eqref{pseudo-bound-1} is bounded by 
\begin{equation*}
\begin{aligned}
\normm{\int \lr{x}^{-12} \lr{k}^{-1} |h(x)| dx}{k} 
  \lesssim \normm{\lr{x}^{-11} h}{x} \normm{\lr{k}^{-1}}{k} \lesssim \normm{\lr{x}^{-11} h}{x},
\end{aligned}
\end{equation*}
where we used the Cauchy-Schwartz inequality for $dx$ integral.
The estimate \eqref{pseudo-bound-2} follows in a similar fashion using \eqref{estimates-partial_k_and_x_m}. 
\end{proof}

\subsection{Distorted Fourier transform}
In this subsection, we present the generalization of the standard Fourier transform
which we apply later to the perturbed Schr\"{o}dinger operator $H = -\partial_x^2 +V$. 
Throughout the section and further in the paper, 
we use the notation
\begin{align*}
\wD := \sqrt{-\partial_x^2+V}, \qquad D := \sqrt{-\partial_x^2} = |\partial_x|. 
\end{align*}
The standard Fourier transform on the line is defined for every $\phi \in L^2$ as 
$$
\widehat{\mathcal{F}} \phi(k) = \widehat{\phi}(k) := \frac{1}{\sqrt{2\pi}} \int e^{-ikx} \phi(x)\, dx,
$$
with
$$
\widehat{\mathcal{F}}^{-1} \phi(x) := \frac{1}{\sqrt{2\pi}} \int e^{ikx} \phi(k)\, dx.
$$
When dealing with the operator $H$, the function $e^{-ikx}$ 
inside the standard Fourier transform can be replaced by a generalized 
eigenfunction $\psi(x, k)$ given in terms of the 
Jost functions $f_\pm$ from \eqref{ode-Jost} in the following way: 
\begin{equation}
\label{wave-function}
\psi(x,k) := \frac{1}{\sqrt{2\pi}} \left\{\begin{array}{l}
\begin{aligned}
&T(k) f_+(x,k) \quad & \text{for } k \geq 0, \\
&T(-k) f_-(x, -k) \quad & \text{for } k < 0
\end{aligned}
\end{array} \right.
\end{equation}
and we now introduce the \textit{distorted} Fourier transform for $h \in \mathcal{S}$ by 
\begin{equation}
\label{distorted-fourier}
\widetilde{\mathcal{F}} h(k) = \widetilde{h}(k) := \int \overline{\psi(x,k)} h(x) \, dx.
\end{equation}

\begin{lemma}
\label{lemma-distorted-fourier}
In above settings, $\widetilde{\mathcal{F}}: L^2_x \to L^2_k$ is an isometry,
\begin{equation}
\label{isometry}
\|\distF h \|_{L^2_k} = \|P_c h \|_{L^2_x},\quad \text{for all  } h \in L^2.
\end{equation}
The inverse of the distorted Fourier transform is given by
$$
\distF^{-1} g (x) = \int \psi(x,k) g (k) dk.
$$
Setting $\wD := \sqrt{-\prtl_{xx} +V}$, we have 
$$
m(\wD)P_c = \distF^{-1} m(k) \distF.
$$
Moreover, the distorted Fourier transform satisfies the following properties:
\begin{itemize}
\item $\distF \distF^{-1} = \text{\rm Id}_{L_k^2}$ and $\distF^{-1} \distF = P_c$;
\item if $h \in L^1_x$, then $\widetilde{h}$ is a continuous, bounded function;
\item since $V$ is generic, $\widetilde{h}(0) = 0$;
\item we have
\begin{equation}
\label{derivative-in-fourier-bound}
\|\prtl_k \widetilde{h}\|_{L^2_k} \lesssim \|\lr{x} h\|_{L^2_x},
\end{equation}
and, for $s \geq 0$
\begin{equation}
\label{k-s-h-tilde-bound}
{\big\| \lr{k}^s \widetilde{h} \big\|}_{L^2_k} = \n{P_c h}_{H_x^s} 
\end{equation}
\end{itemize}
\end{lemma}

\begin{proof}
The proofs can be found in Agmon \cite{A75}, Dunford-Schwartz \cite{DS88}, Germain-Pusateri-Rousset \cite{GPR18} and Yafaev \cite{Y10}.
\end{proof}

\subsection{Decomposition of the generalized eigenfunctions}
By the definition of the generalized eigenfunction we have  
\begin{equation}
\label{gen-eigenfunction-1}
\begin{aligned}
&\text{for  } k\geq 0, \qquad \sqrt{2\pi} \psi(x, k) = T(k) f_+(x, k) = T(k) m_+(x, k) e^{ikx} \\[2pt] 
&\text{for  } k<0, \qquad \sqrt{2\pi} \psi(x, k) = T(-k) f_-(x, -k) = T(-k) m_-(x, -k) e^{ikx}
\end{aligned}
\end{equation}
We introduce smooth versions of the indicator functions, $\sigma_+$ and $\sigma_-$,
\begin{equation}
\label{sigma-functions}
\sigma_+(x) = \int_{-\infty}^x \varphi(y) dy, \quad \text{   and   } \quad \sigma_+(x) + \sigma_-(x) = 1,
\end{equation}
where $\varphi$ is a smooth, nonnegative function, which is equal to $1$ in a neighborhood of $0$, vanishes outside of $[-2,2]$ and $\int \varphi ~dx = 1$. By construction, $\sigma_+(x)$ vanishes for large negative values of $x$ while $\sigma_-(x)$ vanishes for large positive values of $x$. These functions allow to decompose expressions in \eqref{gen-eigenfunction-1} as 
\begin{equation*}
\begin{aligned}
&\text{for  } k\geq0, \quad \sqrt{2\pi} \psi(x, k) = \sigma_+(x) T(k) f_+(x, k) + \sigma_-(x) \left[ f_-(x, -k) +R_-(k) f_-(x,k) \right], \\[2pt] 
&\text{for  } k<0, \quad \sqrt{2\pi} \psi(x, k) = \sigma_-(x) T(-k) f_-(x, -k) + \sigma_+(x) \left[ f_+(x, k) +R_+(-k) f_+(x,-k) \right],
\end{aligned}
\end{equation*}
where we used the relations \eqref{jost-functions-via-scat-par}. 
We further rewrite the generalized eigenfunction above using $m_\pm (x,k) = f_\pm (x,k) e^{\mp ikx}$,
\begin{equation}
\label{gen-eigenfunction-2}
\begin{aligned}
&\text{for  } k\geq0, \qquad
\begin{aligned}
\sqrt{2\pi} \psi(x, k) = &~ \sigma_+(x) T(k) m_+(x, k) e^{ikx} \\
& + \sigma_-(x) \left[ m_-(x, -k) e^{ikx} + R_-(k) m_-(x,k) e^{-ikx} \right]
\end{aligned}\\
&\text{for  } k<0, \qquad
\begin{aligned}
\sqrt{2\pi} \psi(x, k) =&~ \sigma_-(x) T(-k) m_-(x, -k) e^{ikx}\\ 
& + \sigma_+(x) \left[ m_+(x, k)e^{ikx} +R_+(-k) m_+(x,-k) e^{-ikx} \right],
\end{aligned}
\end{aligned}
\end{equation}

\begin{remark}
Based on the boundedness of $\sigma_{\pm} m_{\pm}$ in \eqref{estimates-partial_k^s_m}, $T(k)$ and $R_\pm(k)$ in \eqref{uniform-estimate-T-R},
\begin{equation}
\label{psi-bound}
|\psi(x, k)| \lesssim 1.
\end{equation}
\end{remark}

In the view of \eqref{estimates-partial_k^s_m}, we even further decompose $\psi$ into a singular and regular parts as follows
\begin{equation}
\label{generalized-eigenf-decomp}
\sqrt{2\pi} \psi(x,k) = \psi_S(x, k) + \psi_R(x,k).
\end{equation}
The singular part $\psi_S$ is given by
\begin{equation}
\label{psi-singular-part-decomposition}
\psi_S(x,k) := \sigma_+(x) \psi_{S, +}(x,k) + \sigma_-(x) \psi_{S, -}(x,k),
\end{equation}
where, setting the indicator functions $\mathbf{1}_+ := \mathbf{1}_{[0, \infty)}$ and $\mathbf{1}_- := \mathbf{1}_{(-\infty, 0)}$,
\begin{equation}
\label{singular-part-pm}
\left\{ \begin{array}{l}
\psi_{S,+}(x,k) := \mathbf{1}_+(k) T(k) e^{ikx} + \mathbf{1}_-(k) \left[ e^{ikx} + R_+(-k) e^{-ikx} \right],\\[4pt]
\psi_{S,-}(x,k) := \mathbf{1}_+(k) \left[ e^{ikx} + R_-(k) e^{-ikx} \right] + \mathbf{1}_-(k) T(-k) e^{ikx}.
\end{array} \right.
\end{equation}
The rest is a regular part $\psi_R(x,k)$ given by 
\begin{equation}
\label{regular-part}
\psi_R(x,k) := \left\{ \begin{array}{l}
\begin{aligned}
& \sigma_+(x) T(k) (m_+(x, k)-1) e^{ikx}\\[-2pt]
&+ \sigma_-(x) [ (m_-(x, -k)-1) e^{ikx} + R_-(k) (m_-(x,k)-1) e^{-ikx} ], \qquad k\geq0
\end{aligned}\\
\phantom{t}\\
\begin{aligned}
& \sigma_-(x) T(-k) (m_-(x, -k)-1) e^{ikx} \\[-2pt]
&+ \sigma_+(x) [ (m_+(x, k)-1) e^{ikx} +R_+(-k) (m_+(x,-k)-1) e^{-ikx} ], \qquad k<0. 
\end{aligned}
\end{array} 
\right.
\end{equation}
Note that the regular part is localized around $x=0$ due to the estimates \eqref{estimates-partial_k^s_m}.
On the other hand, the singular part $\psi_S$ 
is a combination of oscillatory terms given by exponential functions with bounded coefficients, 
see \eqref{singular-part-pm}. 
From \eqref{singular-part-pm}, we can further rewrite the elements of the singular part as 
\begin{equation}
\label{singular-part-pm-2}
\left\{ \begin{array}{l}
\psi_{S,+}(x,k) := a_{+,+}(k) e^{ikx} + a_{+,-}(k) e^{-ikx},
\\[4pt]
\psi_{S,-}(x,k) := a_{-,+}(k) e^{ikx} + a_{-,-}(k) e^{-ikx},
\end{array} \right.
\end{equation}
where 
\begin{equation}
\label{a-functions}
\begin{aligned}
& a_{+,+}(k) = \mathbf{1}_+(k) T(k) + \mathbf{1}_-(k), \qquad a_{+,-}(k) = \mathbf{1}_-(k) R_+(-k),
\\
& a_{-,+}(k) = \mathbf{1}_+(k) + \mathbf{1}_-(k) T(-k), \qquad a_{-,-}(k) = \mathbf{1}_+(k) R_-(k).
\end{aligned}
\end{equation}
As it would be useful later, we also point out that, given $h \in P_c L^2_x$, 
the decomposition \eqref{generalized-eigenf-decomp} allows to rewrite 
\begin{equation}
\label{h-sing-reg-decomp}
h = h_S + h_R, \quad \text{where  } h_* := \frac{1}{\sqrt{2\pi}} \int \psi_*(x, k) \widetilde{h}(k) dk \quad \text{  for  } * \in \{S, R\}.
\end{equation}

\begin{remark}
The outcomes of Lemma \ref{lemma-pseudo-bound-1} 
also apply to $\psi_R$. In other words, based on the construction of $\psi_R$ in \eqref{regular-part}, we can verify that 
\begin{equation}
\label{pseudo-psi-regular-bound}
\normm{\int \overline{\psi_R(x,k)} h(x) dx}{k} \lesssim \normm{\lr{x}^{-11} h}{x},
\end{equation}
and 
\begin{equation}
\label{pseudo-dk-psi-regular-bound}
\normm{\int \partial_k \overline{\psi_R(x,k)} h(x) dx}{k} \lesssim \normm{\lr{x}^{-10} h}{x}.
\end{equation}
Moreover, the function $\psi_R$ decays as 
\begin{equation}
\label{psi-regular-decay}
|\psi_R(x,k)| \lesssim \lr{x}^{-12} \lr{k}^{-1}.
\end{equation}
\end{remark}

\subsection{Linear estimates and local decay}
In this section we provide linear and local decay estimates for the term 
\begin{equation}
\label{g-profile-linear-term}
\lr{\wD}^{-1} e^{\pm it \lr{\wD}} P_c f,
\end{equation}
where $\lr{\wD} = \sqrt{H+1} = \distF^{-1} \lr{k} \distF$. 
These estimates will be used in later sections to bound 
the term $\chi(t,x)$ in \eqref{system-eq}. 


\begin{lemma}
\label{lemma-dispersive-estimates}
Under Assumption \ref{V-assumption} on $V$, for every $t\geq 0$ and every $N=0,1,2$ we have
\begin{equation}
\label{dispersive-estimates}
\left\| 
\partial_x^N \lr{\wD}^{-1} e^{\pm it \lr{\wD}} P_c f \right\|_{L_x^\infty} \lesssim \lr{t}^{-\frac{1}{2}+\frac{N}{5}} \left( \|\partial_k \widetilde f\|_{L^2_k} + \|\lr{k}^2 \widetilde f\|_{L^2_k} \right)
\end{equation}
\end{lemma}

The proof of Lemma \ref{lemma-dispersive-estimates} 
is standard and follows from similar, in fact easier, arguments to those in \cite[Proposition 3.10]{GP20}
since here we are not dealing with any bad frequency.
In particular, using 
also the bound $\norm{\lr{k} h}{k}{\infty} \lesssim \norm{\partial_k h}{k}{2} + \norm{\lr{k}^2 h}{k}{2}$,
one can deduce from \cite[Proposition 3.10]{GP20} that
\begin{equation}
\label{dispersive-estimates0}
\left\|
\lr{\wD}^{-1} e^{\pm it \lr{\wD}} P_c f \right\|_{L_x^\infty} 
  \lesssim \lr{t}^{-\frac{1}{2}} \left( \|\partial_k \widetilde f\|_{L^2_k} + \|\lr{k}^2 \widetilde f\|_{L^2_k} \right).
\end{equation}
The estimate \eqref{dispersive-estimates} with $N=1,2$ can be proven with the same argument,
or also using Gagliardo-Nirenberg-Sobolev interpolation and 
\eqref{k-s-h-tilde-bound}.

Note that in \cite{GP20} the decay rates (for $N=0$) associated to the $L^2$ norms 
on the right-hand side of \eqref{dispersive-estimates} are faster, 
and needed to compensate the small growth of these norms 
over the time. 
In our case, we do not have such a growth, so \eqref{dispersive-estimates}
suffices for our purposes. 

\begin{remark}
Lemma \ref{lemma-dispersive-estimates} can be repeated for $V=0$. 
In such case, $\wD$ is replaced by $D$ and the distorted Fourier 
transform becomes the standard Fourier transform. The corresponding 
estimate is
\begin{equation}
\label{dispersive-estimates-standard-fourier}
\left\| 
\lr{D}^{-1} e^{\pm it \lr{D}} f \right\|_{L_x^\infty} \lesssim \lr{t}^{-1/2} \left( \|\partial_k \widehat f\|_{L^2_k} + \|\lr{k}^2 \widehat f\|_{L^2_k} \right).
\end{equation}
\end{remark}

\smallskip
Next, we state the local 
decay estimates which will be used in later sections.
We start with the estimates for the standard Fourier transform under additional assumption at $k=0$.

\begin{lemma}
Suppose $\hat h \in H_k^1$ and $\hat h (0) = 0$, then for every $t \geq 0$ we have
\begin{equation}
\label{local-l-infty-decay-standard-Fourier}
\left\| \lr{x}^{-1} \lr{D}^{-1} e^{\pm it \lr{D}} h \right\|_{L^\infty_x} 
  \lesssim \lr{t}^{-3/4} \| \widehat{h} \|_{H^1_k}. 
\end{equation}
\end{lemma}

The proof of \eqref{local-l-infty-decay-standard-Fourier} is standard, 
and can be found in \cite[Lemma 3.6]{CP21} in the completely analogous case of the Schr\"odinger equation.

Estimates for the distorted Fourier transform have better decays rates under genericity assumption on $V$.

\begin{lemma}
\label{local-l-infty-decay-lemma}
Suppose $V$ satisfies Assumption \ref{V-assumption}. Then, for every $t \geq 0$, we have
\begin{equation}
\label{local-l-infty-decay}
\left\| \lr{x}^{-2} \lr{\wD}^{-1} e^{\pm it \lr{\wD}} P_c h \right\|_{L^\infty_x} 
  \lesssim \lr{t}^{-1} \| \widetilde{h} \|_{H^1_k}. 
\end{equation}
\end{lemma}

The proof of this result is also standard and can be found in \cite{KGVSim}.
An analogous estimate for the Schr\"odinger flow is given in \cite[Lemma 3.6]{CP21}; 
see also the references in \cite{KGVSim,CP21} for other related and more general estimates.

We can extend the result of the above lemma to cases when we deal with the singular
and regular parts of the function $\lr{\wD}^{-1} e^{\pm it \lr{\wD}} P_c h$ 
according to the decomposition in \eqref{h-sing-reg-decomp}. 
First, we have the following analogue of Lemma \ref{local-l-infty-decay-lemma}:

\begin{lemma}
\label{lemma-local-l-infty-decay-singular}
Suppose $V$ satisfies Assumption \ref{V-assumption}. Then, for every $t \geq 0$, we have
\begin{equation}
\label{local-l-infty-decay-singular}
\left\| \lr{x}^{-2} \left( \lr{\wD}^{-1} e^{\pm it \lr{\wD}} P_c h \right)_S \right\|_{L^\infty_x}
  \lesssim \lr{t}^{-1} \| \widetilde{h} \|_{H^1_k},
\end{equation}
\end{lemma}

For the regular part, the function $\psi_R$ is already decaying according 
to \eqref{psi-regular-decay}, which allows to get estimates with 
growing weight.

\begin{lemma}
Suppose $V$ satisfies Assumption \ref{V-assumption}. Then, for every $t \geq 0$, we have
\begin{equation}
\label{local-l-infty-decay-regular}
\left\| \lr{x}^{10} \left( \lr{\wD}^{-1} e^{\pm it \lr{\wD}} P_c h \right)_R \right\|_{L^\infty_x} 
  \lesssim \lr{t}^{-1} \| \widetilde{h} \|_{H^1_k}.
\end{equation}
\end{lemma}

\medskip
\section{Set-up and bootstrap}
In this section we set-up 
the upcoming analysis of the evolution equations 
for $a(t)$ and $\chi(t,x)$ in \eqref{system-eq}, and state our main bootstrap propositions
that will imply Proposition \ref{propmain}.

We let $H:= P_c(-\prtl_{xx} + V(x))$ with $V(x) :=-4Q^3$,
where $Q(x)$ is the soliton \eqref{soliton}. 
Then the operator \eqref{estimates-lin} is $\calL = H+1$. 
Recall the definition of
$g(t, x) := e^{it \sqrt{H+1}}  ( \partial_t - i \sqrt{H+1} ) \chi(t, x)$
in \eqref{g-profile}, so that 
$\chi$ is given as in \eqref{chi-in-terms-of-g}
Recall also the notation for our main norm  \eqref{main-normintro},
\begin{equation}
\label{main-norm}
\|v\|_{X} := \|\lr{t}^2 a(t)\|_{L^\infty_t (0, \infty)} + 
\|\prtl_k \wg \|_{L^\infty_t ((0, \infty): L_k^2)} + \|\lr{k}^2 \wg \|_{L^\infty_t ((0, \infty): L_k^2)}.
\end{equation}

\smallskip
\begin{remark}
\label{remark-chi-estimates}
From Lemma \ref{lemma-dispersive-estimates} and Lemma \ref{local-l-infty-decay-lemma}
the following estimates for $\chi$ follow
\begin{equation}
\label{dispersive-decay-for-chi}
\|\partial_x^N \chi(t)\|_{L_x^\infty} \lesssim \lr{t}^{-\frac{1}{2}+\frac{N}{5}} 
  \left( \|\partial_k \wg\|_{L^2_k} + \|\lr{k}^2 \wg\|_{L^2_k} \right)
  \lesssim \lr{t}^{-\frac{1}{2}+\frac{N}{5}} {\| v \|}_X, \qquad N=0,1,2,
\end{equation}
and 
\begin{equation}
\label{local-decay-for-chi}
\left\| \lr{x}^{-2} \chi(t) \right\|_{L^\infty_x} \lesssim \lr{t}^{-1} \| \wg \|_{H^1_k} \lesssim \lr{t}^{-1} {\| v \|}_X.
\end{equation}
\end{remark}

Applying the distorted Fourier transform to the relation \eqref{g-profile},
we obtain 
$\widetilde g(t, k) = e^{it \lr{k}} ( \partial_t - i \lr{k} ) \widetilde \chi(t, k)$
and, after substitution into the second equation in \eqref{system-eq}, 
\begin{equation}\label{wg-evolution}
\partial_t \wg (t, k) = e^{it\lr{k}} \widetilde{\mathcal{F}} P_c N(Q, v),
\end{equation}
where $N(Q, v) = 6Q^2 v^2 + 4 Q v^3 + v^4$ and we recall that the term $v$ 
is decomposed as $v(t,x) = a(t) \rho(x) + \chi (t, x)$. 
Integrating the equation \eqref{wg-evolution} over time give
\begin{equation}
\label{wg-duhamel}
\wg (t,k) = \wg (0,k) + \int_0^t e^{is\lr{k}} \distF P_c N (s, k) ds,
\end{equation}
and applying the partial derivative with respect to $k$, we obtain
\begin{equation}
\label{prtl-wg-duhamel}
\prtl_k \wg (t,k) = \prtl_k \wg (0,k) + \int_0^t \prtl_k \left[ e^{is\lr{k}} \distF P_c N (s, k) \right] ds.
\end{equation}


In the course of our estimates we will need to deal with $\|\partial_t \wg (t)\|_{L_k^2}$,
which can be bounded using the following 
\begin{lemma}
\label{lemma-prtl-t-wg}
Given $\n{v}_X \ll 1$, we have
\begin{equation}
\label{prtl-t-wg-estimates}
\|\prtl_t \wg\|_{L_k^2} \lesssim \lr{t}^{-3/2} \|v\|_X^2
\end{equation}
\end{lemma}

\begin{proof}
From the evolution equation \eqref{wg-evolution}, since $\distF$ is an $L_x^2 \to L_k^2$ isometry, we have 
\begin{equation*}
\|\prtl_t \wg\|_{L_k^2} \lesssim \|Q^2v^2\|_{L_x^2} + \|Q v^3\|_{L_x^2} + \|v^4\|_{L_x^2}.
\end{equation*}
We substitute the decomposition $v(t, x) = a(t) \rho(x)+ \chi(t,x)$, 
and bound the first two terms using the local decay \eqref{local-decay-for-chi} as
\begin{equation*}
\|Q^2v^2\|_{L_x^2} \lesssim |a(t)|^2 + \|\lr{x}^{-2} \chi\|^2_{L_x^\infty} \lesssim \lr{t}^{-2} \|v\|_X^2,
\end{equation*}
with a similar estimate for $\|Qv^3\|_{L_x^2}$.
For the remaining term we have 
{
\begin{equation*}
\|v^4\|_{L_x^2} \lesssim |a(t)|^4 + \| \chi^4\|_{L_x^2} \lesssim \lr{t}^{-8} \|v\|_X^4 + \| \chi\|_{L_x^2} \| \chi\|^3_{L_x^\infty}.
\end{equation*}
}
The term $\| \chi\|_{L_x^\infty}$ is bounded by \eqref{dispersive-decay-for-chi}, 
while $\|\chi\|_{L_x^2} \lesssim \|\wg \|_{L_k^2} \leq \|v\|_X$, using \eqref{isometry}. 
Then, we have $\|v^4\|_{L_x^2} \lesssim \lr{t}^{-3/2} \|v\|^4_X,$
and combining all the estimates we obtain the desired bound \eqref{prtl-t-wg-estimates}.
\end{proof}

In the following subsections, we give estimates on the three norms in \eqref{main-norm}.

\smallskip
\subsection{Estimates for the first norm in \eqref{main-norm}}

\begin{proposition}
\label{a(t)-norm-estimates}
Given $\n{v}_X \ll 1$, for every $t \geq 0$ we have
\begin{equation}
\label{a-bound}
\lr{t}^2 |a(t)| \lesssim |a(0)| + \|v\|_X^2.
\end{equation}
\end{proposition}

\begin{proof}
Using the integral equation for $a(t)$ in \eqref{integral-eq}, with $F[v]$ as in \eqref{system-eq}, we have
\begin{equation*}
\begin{aligned}
\lr{t}^2 a(t) &= \lr{t}^2 e^{-\Omega t} \left[ a (0) + \frac{1}{2\Omega} \int_0^\infty e^{-\Omega s} F[v](s) ds \right]\\[3pt]
& \quad - 
\frac{\lr{t}^2 e^{-\Omega t}}{2\Omega} \int_0^t e^{\Omega s} F[v](s) ds - 
\frac{\lr{t}^2}{2\Omega} \int_t^\infty e^{-\Omega (s-t)} F[v](s) ds\\[3pt]
& := \mathcal{I}_1 (t) + \mathcal{I}_2(t) + \mathcal{I}_3(t).
\end{aligned}
\end{equation*}
To bound $F[v](s)$ 
we use the definition 
in \eqref{estimates-lin} and the decomposition of $v$ in \eqref{v-decomp} 
to get
\begin{equation}
\label{F-v-bound}
|F[v](s)| \lesssim |\lr{v^2(s), \rho}|+ |\lr{v^3(s), \rho}| + |\lr{v^4(s), \rho}| \lesssim \norm{\lr{x}^{-2}v(s)}{x}{\infty}^2 \lesssim \lr{s}^{-2} \n{v}_X^2,
\end{equation}
where for the last inequality we used \eqref{local-decay-for-chi}.
Then, using \eqref{F-v-bound}, direct computations show that 
\begin{equation*}
|\mathcal{I}_1(t)|\lesssim |a(0)| +\n{v}_X^2, \quad \text{and} \quad |\mathcal{I}_2(t)|, |\mathcal{I}_3(t)|\lesssim \n{v}_X^2.    
\end{equation*}
Combining these estimates, we obtain the desired result.
\end{proof}

\smallskip
\subsection{Estimates for the second norm in \eqref{main-norm}}
Here, we give the $L_k^2$ bound for $\partial_k \wg$. 

\begin{proposition}
\label{proposition-partial-k-L2-norm}
Given $\n{v}_X \ll 1$, for every $t \geq 0$ we have \begin{equation}
\label{prtl-k-wg-norm-bound}
\normm{\partial_k \wg (t)}{k} \lesssim \normm{\partial_k \wg (0)}{k} + \|v\|_X^2.
\end{equation}
\end{proposition}

The proof is based on the estimates for the terms in the Duhamel expression
for $\partial_k \wg$ in \eqref{prtl-wg-duhamel}. Recall that, according to the decomposition $v(s,x) = a(s) \rho(x) + \chi(s,x)$, we have
\begin{equation}
\label{FN_c-decomposition}
\distF N_c (s, k) = \distF P_c \left(6Q^2 (a(s)\rho + \chi(s))^2 +4 Q (a(s)\rho + \chi(s))^3 + (a(s)\rho + \chi(s))^4 \right).
\end{equation}
Then the $L_k^2$ bound for $\partial_k \wg$ is given by 
\begin{equation}
\label{partial-k-wg-intermediate-bound}
\normm{\partial_k \wg(t)}{k} \lesssim \normm{\partial_k \wg(0)}{k} + \normm{A_1(t)}{k} + \normm{A_2(t)}{k} + \normm{A_3(t)}{k}, 
\end{equation}
where 
\begin{equation}
\label{A-j-functions}
\begin{aligned}
& A_1(t,k) := \int_0^t \partial_k \left[e^{is\lr{k}} \distF P_c Q^2 (a(s)\rho + \chi(s))^2(k) \right] ds,\\
& A_2(t,k) := \int_0^t \partial_k \left[e^{is\lr{k}} \distF P_c Q (a(s)\rho + \chi(s))^3(k) \right] ds,\\
& A_3(t,k) := \int_0^t \partial_k \left[e^{is\lr{k}} \distF P_c \left(a(s)\rho + \chi(s)\right)^4(k) \right] ds.
\end{aligned}
\end{equation}
To prove Proposition \ref{proposition-partial-k-L2-norm}, it suffices to show that $\normm{A_j(t)}{k} \lesssim \n{v}_X^2$ for every $j=1,2,3$. We will verify these estimates in the following lemmas as well as in Section \ref{section-quadratic} and Section \ref{section-quartic}. 

\begin{lemma}
\label{lemma-A-1-L2-norm-bound}
Given $\n{v}_X\ll 1$, for every $t\geq 0$ we have 
\begin{equation}
\label{A-1-L2-norm-bound}
\normm{A_1(t)}{k} \lesssim \normm{\int_0^t
s \frac{k}{\lr{k}} e^{is\lr{k}} \distF P_c Q^2 \chi^2(s) ds}{k} + \n{v}_X^2.
\end{equation}
\end{lemma}

\begin{proof}
Distributing $\partial_k$ 
and using Minkowski inequality, we obtain 
\begin{equation*}
\begin{aligned}
\normm{A_1(t)}{k} \lesssim&~ \int_0^t \normm{ \partial_k \distF P_c Q^2 (a(s)\rho + \chi(s))^2 }{k}ds\\
& + \normm{\int_0^t s \frac{k}{\lr{k}} e^{is\lr{k}} \distF P_c Q^2 (a(s)\rho + \chi(s))^2 ds}{k}\\
& =: A_{11} + A_{12}.
\end{aligned}
\end{equation*}
The first term $A_{11}$ is easily bounded using \eqref{derivative-in-fourier-bound},
followed by
the local decay estimates \eqref{local-decay-for-chi}: 
\begin{equation*}
\begin{aligned}
A_{11} &\lesssim \int_0^t \normm{\lr{x} Q^2 (a(s) \rho + \chi(s))^2}{x}ds
\\
& \lesssim \int_0^t a^2(s) ds + \int_0^t \norm{\lr{x}^{-2} \chi(s)}{x}{\infty}^2 ds \lesssim \n{v}_X^2. 
\end{aligned}
\end{equation*}
For the second term $A_{12}$ we bound it similarly, term by term, as follows 
\begin{equation*}
\begin{aligned}
A_{12} & \lesssim \int_0^t s\; a^2(s) \normm{\distF P_c Q^2 \rho^2}{k} ds + \int_0^t s\; |a(s)| \normm{\distF P_c Q^2 \rho \chi(s)}{k} ds\\
& \quad + \normm{\int_0^t
s \frac{k}{\lr{k}} e^{is\lr{k}} \distF P_c Q^2 \chi^2(s)  ds}{k}\\
& := B_1 + B_2 + B_3.
\end{aligned}
\end{equation*}
Using \eqref{isometry} and \eqref{local-decay-for-chi} we obtain
\begin{equation*}
\begin{aligned}
& B_1 \lesssim \n{v}_X^2 \int_0^t \frac{s}{\lr{s}^4} ds \lesssim \n{v}_X^2, 
\\
& B_2 \lesssim \int_0^t s\; |a(s)| \normm{\lr{x}^{-2} \chi(s)}{x} ds \lesssim \n{v}_X^2 \int_0^t \frac{s}{\lr{s}^3} ds \lesssim \n{v}_X^2,
\end{aligned}
\end{equation*}
and $B_3$ is the term which appears in \eqref{A-1-L2-norm-bound}.
This completes the proof. \end{proof}

\begin{remark}
Note that in the proof of Lemma \ref{lemma-A-1-L2-norm-bound}, the term $B_3$ can not be bounded simply
 using Minkowski's inequality, as the corresponding integral in $ds$ would diverge when $t \to \infty$. 
 Indeed, we would have
\begin{equation}
\label{divergent-B3}
B_3 \lesssim \int_0^t s \normm{\lr{x}^{-2} \chi(s)}{k}^2 ds \lesssim \n{v}_X^2 \int_0^t \frac{s}{\lr{s}^2} ds \approx \n{v}_X^2 \log t.
\end{equation}
To get a $O(\n{v}_X^2)$ bound for the term $\normm{A_1(t)}{k}$, 
we need some more refined analysis, which is provided in Section \ref{section-quadratic}.
\end{remark}

\begin{lemma}
\label{lemma-A-2-L2-norm-bound}
Given $\n{v}_X\ll 1$, for every $t\geq 0$ we have 
\begin{equation}\label{A-2-L2-norm-bound}
\normm{A_2(t)}{k} \lesssim  \n{v}_X^3.
\end{equation}
\end{lemma}

\begin{proof}
The proof is similar to the proof of Lemma \ref{lemma-A-1-L2-norm-bound}.
However, unlike in \eqref{divergent-B3}, each term, which appears after applying $\partial_k$ inside the $A_2(t)$ integral
in \eqref{A-j-functions} and expanding the cubic term $(a(s)\rho + \chi(s))^3$, is directly bounded by $\n{v}_X^3$.
\end{proof}

\begin{lemma}
\label{lemma-A-3-L2-norm-bound}
Given $\n{v}_X\ll 1$, for every $t\geq 0$ we have 
\begin{equation}
\label{A-3-L2-norm-bound}
\normm{A_3(t)}{k} \lesssim \normm{\int_0^t \partial_k \left[ e^{is\lr{k}} \distF P_c \chi^4(s) \right] ds}{k} + \n{v}_X^4.
\end{equation}
\end{lemma}

\begin{proof}
The proof is similar to the proof of previous lemmas, using the fact that $\rho(x)$ is an exponentially decaying function
and \eqref{local-decay-for-chi}.
\end{proof}

\begin{remark}
Since there are no (exponentially) decaying function in front of $\chi^4$
we cannot just rely on the local decay estimates \eqref{local-decay-for-chi}
to bound the quartic terms on the left-hand side of \eqref{A-3-L2-norm-bound}.
We will bound these terms in Section \ref{section-quartic} ,
using some multilinear analysis in the spirit of \cite{CP21,GP20}.
\end{remark}

The proof of Proposition \ref{proposition-partial-k-L2-norm} is now 
straightforward assuming that both the quadratic term in \eqref{A-1-L2-norm-bound} and the 
quartic term in \eqref{A-3-L2-norm-bound} are bounded by $\n{v}_X^2$. 
These two bounds will be proven, respectively, in the upcoming Sections \ref{section-quadratic} and \ref{section-quartic}. 

\begin{proof}[Proof of Proposition \ref{proposition-partial-k-L2-norm}]
Under the above assumption, Lemmas \ref{lemma-A-1-L2-norm-bound}--\ref{lemma-A-3-L2-norm-bound}
imply
\begin{equation*}
\normm{A_1(t)}{k} + \normm{A_2(t)}{k} + \normm{A_3(t)}{k} \lesssim \n{v}_X^2.
\end{equation*}
This bound applied to \eqref{partial-k-wg-intermediate-bound} completes the proof. 
\end{proof}

\smallskip
\subsection{Estimates for the third norm in \eqref{main-norm}}
Here we give the $L_k^2$ bound for $\normm{\lr{k}^2 \wg}{k}$.

\begin{proposition}
\label{proposition-k2-wg-norm-bound}
For every $t \geq 0$ we have 
\begin{equation}
\label{k2-wg-norm-bound}
\normm{\lr{k}^2 \wg (t)}{k} \lesssim \normm{\lr{k}^2 \wg (0)}{k} + \|v\|_X^2.
\end{equation}
\end{proposition}

To prove the above statement, we use the Duhamel's expression \eqref{wg-duhamel}, 
apply Minkowski and recall 
\eqref{estimates-lin}  
to get
\begin{equation}
\label{k2-wg-bound-decomposition}
\begin{aligned}
\normm{\lr{k}^2 \wg (t)}{k} & \lesssim \normm{\lr{k}^2 \wg (0)}{k} + \int_0^t \normm{ \lr{k}^2 \distF P_c Q^2 v^2(s)}{k} ds+ \int_0^t \normm{ \lr{k}^2 \distF P_c Q v^3(s)}{k} ds\\
& \quad + \int_0^t \normm{ \lr{k}^2 \distF P_c v^4(s)}{k} ds\\
& := \normm{\lr{k}^2 \wg (0)}{k} + \mathcal{I}_1(t) + \mathcal{I}_2(t) + \mathcal{I}_3(t).
\end{aligned}
\end{equation}
Therefore, it suffices to show the following:

\begin{lemma}
\label{lemma-k2-wg-I-j-bound}
For every $t \geq 0$ we have 
\begin{equation}
\label{k2-wg-I-j-bound}
\mathcal{I}_j (t) \lesssim \n{v}_X^2 \quad \text{for every  } j=1,2,3.
\end{equation}
\end{lemma}

\begin{proof}
Starting with the quadratic terms, we recall the definition of the norm \eqref{main-norm},
and estimate using \eqref{k-s-h-tilde-bound}
followed by the pointwise and local decay estimates \eqref{dispersive-decay-for-chi} and \eqref{local-decay-for-chi}:
\begin{align*}
\normm{ \lr{k}^2 \distF P_c Q^2 v^2(s)}{k} & \lesssim {\| Q^2 v^2(s) \|}_{H_x^2}
  \lesssim {\| Q v(s) \|}_{H_x^2} {\| Q v(s) \|}_{L_x^\infty}
\\ 
& \lesssim \sup_{j=0,1,2} {\| \langle x\rangle^{-5} \partial_x^j v(s) \|}_{L_x^2} {\| Q v(s) \|}_{L_x^\infty}
  \lesssim  \langle s \rangle^{-1/10}\n{v}_X \langle s \rangle^{-1} \n{v}_X,
\end{align*}
which proves the desired bound for $j=1$.
The estimates for the cubic term $\mathcal{I}_2$ is simpler so we can skip it.
For $\mathcal{I}_3$
we use the decay \eqref{dispersive-decay-for-chi} to see that
\begin{align*}
\normm{ \lr{k}^2 \distF P_c v^4(s)}{k} \lesssim {\big\| v^4(s) \big\|}_{H_2^x}
  \lesssim {\big\| v(s) \big\|}_{H_x^2} {\big\| v(s) \big\|}_{L^\infty_x}^3
  \lesssim \langle s \rangle^{-3/2} \n{v}_X^4.
\end{align*}
\end{proof}



\medskip
\section{Quadratic term analysis} 
\label{section-quadratic}
In this section we prove the following bound for the quadratic 
terms on the right-hand side of \eqref{A-1-L2-norm-bound} that were not treated in the previous section.

\begin{lemma}
\label{lemma-quadratic-term-bound}
Given $\n{v}_X \ll 1$, for every $t \geq 0$ we have 
\begin{equation}
\label{estimate-lemma-quadratic-bound}
\normm{\int_0^t s \frac{k}{\lr{k}} e^{is\lr{k}} \distF P_c Q^2 \chi^2(s) ds}{k} \lesssim \n{v}_X^2.
\end{equation}
\end{lemma}

To get a bound 
in \eqref{estimate-lemma-quadratic-bound}, we use the following smoothing-type estimate:

\def\jx{\lr{x}}

\begin{lemma}\label{lemsmoothing}
Given a function $\mathcal{Q}: \R_x\times \R_k \mapsto \C$ such that, for some $\beta \in \R$,
\begin{equation*}
\sup_{x\in\R,\,k\in\R} \big| \jx^{\beta} \mathcal{Q}(x,k) \big| < \infty,
\end{equation*}
and $\phi: \R \mapsto \C$ such that $|\phi(k)|\lesssim |k|/\lr{k}$,
then, for all $t\geq0$,
\begin{equation}\label{smoothinghom}
{\Big\| \jx^{\beta} \int_\R \mathcal{Q}(x,k) 
  \phi(k) e^{i\lr{k}s} h(k)\,dk \Big\|}_{L_x^\infty L_s^{2}([0,t])} \lesssim {\| h \|}_{L^2}.
\end{equation}
Under the same assumptions, the following inhomogeneous estimate holds:
\begin{align}\label{smoothinginhom}
{\Big\| \int_0^t \left[ \int_\R e^{-i \lr{k}s} \overline{\phi}(k) 
  \overline{\mathcal{Q}}(x,k)F(s,x)\,dx \right] \,ds \Big\|}_{L_k^2}
\lesssim & {\Big\|\jx ^{-\beta}F \Big\|}_{L_x^1 L_s^2([0,\infty])}.
\end{align}
\end{lemma}

\smallskip
We now complete the proof of Lemma \ref{lemma-quadratic-term-bound} using Lemma \ref{lemsmoothing} above.

\begin{proof}[Proof of Lemma \ref{lemma-quadratic-term-bound}]
Applying \eqref{smoothinginhom} 
to the 
left hand side of \eqref{estimate-lemma-quadratic-bound}, that is, 
letting $\beta=0$, $\mathcal{Q}(x, k) = \psi(x,k)$, $\phi(k) = k/\lr{k}$, 
and $F(s,x) = s \,Q^2(x) \chi^2(s,x)$, we get
\begin{align}\label{Destimate}
\begin{split}
{\Big\| \int_0^t \left[ \int_\R e^{i \lr{k} s} \frac{k}{\lr{k}}
\overline{\psi(x,k)} F(s,x) \,dx \right] \,ds \Big\|}_{L_{k}^{2}} & \lesssim {\| s \, Q^2(\cdot) \chi^2(s,\cdot) \|}_{L_x^1(\R) L_s^2([0,\infty])}
  \\
  & \lesssim {\| s \, Q(\cdot) \chi^2(s,\cdot) \|}_{L_s^2([0,\infty]) L^\infty_x(\R)}
  \\
  & \lesssim {\| v \|}_X^2 {\| s \lr{s}^{-2} \|}_{L_s^2([0,\infty])} \lesssim {\| v \|}_X^2,
\end{split}
\end{align}
having used the local decay estimate \eqref{local-decay-for-chi}.
\end{proof}

\smallskip
Observe that Lemma \ref{lemsmoothing} does not use the genericity of the potential.
The proof is similar to the one for the analogous 
case of the Schr\"odinger flow $e^{itH}$ given in \cite[Lemma 3.5]{CP22}.
We provide the details here for completeness.

\begin{proof}[Proof of Lemma \ref{lemsmoothing}]
Without loss of generality we may assume $\beta=0$ 
and restrict the integral in \eqref{smoothinghom} to $k\in[0,\infty)$.
Making a change of variable $\lr{k} = \lambda$, 
so that $k\lr{k}^{-1} dk = d\lambda$, $dk = \frac{\lambda}{\sqrt{\lambda^2-1}} d\lambda$ one has
\begin{align*}
I(s,x) & := \int_0^\infty \mathcal{Q}(x,k)\phi(k)e^{i \lr{k} s}h(k)\,dk
  \\
  & = \int_1^{\infty}\mathcal{Q}(x,\sqrt{\lambda^2-1})\phi(\sqrt{\lambda^2-1})
  e^{is\lambda} h(\sqrt{\lambda^2-1}) \frac{\lambda}{\sqrt{\lambda^2-1}}\,d\lambda.
\end{align*}

Taking the $L^2_s$ norm and applying Plancherel's theorem we obtain
\begin{align*}
{\| I(\cdot,x) \|}_{L^2_s([0,t])}^2 
  & \lesssim \left\| \int_\R \mathcal{Q}(x,\sqrt{\lambda^2-1}) \, \mathbf{1}_{[1,\infty)}(\lambda)
  \phi(\sqrt{\lambda^2-1})e^{is\lambda} h(\sqrt{\lambda^2-1}) \frac{\lambda}{\sqrt{\lambda^2-1}}
   \,d\lambda \right\|_{L^2_s(\R)}^2
\\
& \lesssim \int_1^{\infty} \left|  \mathcal{Q}(x,\sqrt{\lambda^2-1})
  \frac{\lambda \, \phi(\sqrt{\lambda^2-1})}{\sqrt{\lambda^2-1}} 
  h(\sqrt{\lambda^2-1}) \right|^{2} \,d\lambda 
\\
  & \lesssim \int_1^{\infty} \big| h(\sqrt{\lambda^2-1}) \big|^2 \,d\lambda 
  \lesssim \int_0^{\infty} | h (k) |^{2} \frac{k}{\lr{k}} \,dk
  \lesssim {\| h \|}_{L^2}^2 .
\end{align*}

To obtain the inhomogeneous estimate \eqref{smoothinginhom}
we test the expression on the left-hand side against an arbitrary function $h \in L^2$:
\begin{align*}
& \left \langle  h,\int_0^t \left[ \int_\R e^{-i \lr{k}s} \overline{\phi}(k) 
\overline{\mathcal{Q}}(y,k)F(s,y)\,dy\right] ds \right \rangle_{L^2_k}
\\
& = \int_0^t \int_\R \Big( \int_\R e^{i \lr{k}s}\phi(k)\mathcal{Q}(y,k) h(k)\,dk \Big)  \overline{F}(s,y)\,dyds
\end{align*}
Applying H\"older's inequality followed by the homogeneous estimate \eqref{smoothinghom}:
\begin{align*}
& \left| \int_0^t \int_\R \Big( \int e^{i\lr{k}s}\phi(k)\mathcal{Q}(y,k) h(k)\,dk \Big) \overline{F}(s,y) \,dyds\right|
\\
& \lesssim {\Big\|
  \int\mathcal{Q}(x,k)\phi(k)e^{i\lr{k} s}h(k)\,dk \Big\|}_{L_x^\infty L_s^2([0,t])} 
  \big\| F \big\|_{L_x^1 L_s^2([0,t])}
\\
& \lesssim {\| h \|}_{L^2} {\big\| F \big\|}_{L_x^1 L_s^2([0,t])}
\end{align*}
which implies \eqref{smoothinginhom} by duality.
\end{proof}

\medskip
\section{Quartic term analysis} 
\label{section-quartic}
In this section we provide a proof of the following lemma:

\begin{lemma}
\label{lemma-quartic-term-bound}
Given $\n{v}_X \ll 1$, for every $t \geq 0$ we have 
\begin{equation}
\label{quartic-term-bound}
\left\| \int_0^t \partial_k \left( e^{is\lr{k}} \distF 
  \chi^4(s) \right) ds \right\|_{L_k^2} \lesssim \n{v}_X^4.
\end{equation}
\end{lemma}

This lemma gives a bound for the quartic term in \eqref{A-3-L2-norm-bound}
and, together with Lemmas \ref{lemma-A-1-L2-norm-bound}, \ref{lemma-A-3-L2-norm-bound} and \ref{lemma-quadratic-term-bound},
will thus complete the proof of Proposition \ref{proposition-partial-k-L2-norm}.
This in turn gives the main theorem.

\medskip
The term inside the integral in \eqref{quartic-term-bound} can be written as 
\begin{equation}
\label{quartic-integral}
\begin{aligned}
& 16 e^{is\lr{k}} \distF 
  \chi^4(s,k) \\
&=  
\sum_{\iota_1, \iota_2, \iota_3, \iota_4 \in \{+, -\}} \iota_1 \iota_2 \iota_3 \iota_4 \int \frac{e^{is \Phi_{\iota_1 \iota_2 \iota_3 \iota_4} 
}}{\lr{l} \lr{m} \lr{n} \lr{p}} \wg^{(\iota_1)}(l) \wg^{(\iota_2)}(m) \wg^{(\iota_3)}(n) \wg^{(\iota_4)}(p) 
\mu_{\iota_1 \iota_2 \iota_3 \iota_4}
dlmnp
\end{aligned}
\end{equation}
where we used the short-hand notation $dlmnp := dldmdndp$,
the definition of the distorted Fourier transform
and the relation \eqref{chi-in-terms-of-g}, with the functions
$\Phi_{\iota_1 \iota_2 \iota_3 \iota_4}$ and $\mu_{\iota_1 \iota_2 \iota_3 \iota_4}$,
depending on the variables $k,l,m,n,p$, defined as follows
\begin{equation}
\label{Phi-quartic}
\Phiiota = \lr{k} - \iota_1 \lr{l} - \iota_2 \lr{m} - \iota_3 \lr{n} - \iota_4 \lr{p},
\end{equation}
and 
\begin{equation}
\label{spectral-distribution}
\muiota = 
\int \overline{\psi(x, k)} \psi^{(\iota_1)}(x, l) \psi^{(\iota_2)}(x, m) 
  \psi^{(\iota_3)}(x, n) \psi^{(\iota_4)}(x, p) dx.
\end{equation}
We call $\muiota$ as \textit{the nonlinear spectral distribution}. 
We further decompose this spectral distribution into singular 
and regular parts according to \eqref{generalized-eigenf-decomp},
and consider each part separately.

\smallskip
\subsection{Decomposition of the nonlinear spectral distribution}
Using the decomposition \eqref{generalized-eigenf-decomp}, 
we rewrite the nonlinear spectral distribution (\ref{spectral-distribution}) as
$$
(2 \pi)^{5/2} \muiota = \mu_{\iota_1 \iota_2 \iota_3 \iota_4}^S(k, l, m, n, p) 
  + \mu_{\iota_1 \iota_2 \iota_3 \iota_4}^R(k, l, m, n, p),
$$
where $\mu^S$ and $\mu^R$ denote the `Singular' and `Regular' parts, respectively. 
Below, we give the exact construction of each part. 
The singular part is the sum of integrals of functions which oscillate 
at either $-\infty$ or $+ \infty$. The regular part consists, instead, of integrals of 
localized or compactly supported functions. 

Recalling \eqref{psi-singular-part-decomposition},
we define the singular part by
\begin{equation}
\label{spectral-singular-decomp}
\begin{aligned}
& \mu_{\iota_1 \iota_2 \iota_3 \iota_4}^S(k, l, m, n, p) 
  = \mu_{\iota_1 \iota_2 \iota_3 \iota_4}^{+}(k, l, m, n, p)
  + \mu_{\iota_1 \iota_2 \iota_3 \iota_4}^{-}(k, l, m, n, p),
  \\[2pt]
& \mu_{\iota_1 \iota_2 \iota_3 \iota_4}^{*}(k, l, m, n, p) =
  \int \nu_*(x) \overline{\psi_{S,*}(x,k)} \psi^{(\iota_1)}_{S,*}(x,l)
  \psi^{(\iota_2)}_{S,*}(x,m) \psi^{(\iota_3)}_{S,*}(x,n) \psi^{(\iota_4)}_{S,*}(x,p) dx,
\end{aligned}
\end{equation}
where $\nu_*(x) = \sigma_*^5(x)$ for $* \in \{+, -\}$. 

For the regular part, we define the set
\begin{equation*}
Z_R := \{(A_1, A_2, A_3, A_4, A_5): \exists j = 1,2,3,4,5 \text{ s.t. } A_j = R\} \subseteq \{S, R\}^5,
\end{equation*}
and write $\mu^R$ as the sum of two parts as follows:
\begin{equation*}
\mu^R_{\iota_1 \iota_2 \iota_3 \iota_4} (k,l,m,n,p) 
  = \mu^{R,1}_{\iota_1 \iota_2 \iota_3 \iota_4} (k,l,m,n,p)
  + \mu^{R,2}_{\iota_1 \iota_2 \iota_3 \iota_4} (k,l,m,n,p)
\end{equation*}
where
\begin{equation}
\label{mu-r-1}
\mu^{R,1}_{\iota_1 \iota_2 \iota_3 \iota_4} = \sum_{(A,B,C,D,E) \in Z_R} 
\int \overline{\psi_{A}(x,k)} \psi^{(\iota_1)}_{B}(x,l)
\psi^{(\iota_2)}_{C}(x,m) \psi^{(\iota_3)}_{D}(x,n)
\psi^{(\iota_4)}_{E}(x,p)dx,
\end{equation}
and 
\begin{equation}
\label{mu-r-2}
\mu^{R,2}_{\iota_1 \iota_2 \iota_3 \iota_4} = \int \overline{\psi_{S}(x,k)} \psi^{(\iota_1)}_{S}(x,l)
\psi^{(\iota_2)}_{S}(x,m) \psi^{(\iota_3)}_{S}(x,n)
\psi^{(\iota_4)}_{S}(x,p)dx - \mu^S_{\iota_1 \iota_2 \iota_3 \iota_4}.
\end{equation}
Notice that $\mu^{R,1}$ is the sum of integrals of functions localized in $x$ around $0$, 
and $\mu^{R,2}$ consists of integrals of functions compactly supported around $0$
due to the presence of product of the form $\sigma_+(x) \sigma_-(x)$. 

We decompose accordingly \eqref{quartic-integral} and write 
\begin{equation*}
16 e^{is\lr{k}} \distF \chi^4(s,k) 
  = \mathcal{N}_+ (s,k) + \mathcal{N}_{-}(s,k) + \mathcal{N}_{R,1}(s,k) + \mathcal{N}_{R,2}(s,k),
\end{equation*}
where 
\begin{equation}
\label{general-N-term}
\begin{aligned}
&(2 \pi)^{5/2}\mathcal{N}_*(s,k) 
\\
&= \sum_{\iota_1, \iota_2, \iota_3, \iota_4 \in \{+, -\}} \iota_1 \iota_2 \iota_3 \iota_4 \int \frac{e^{is \Phi_{\iota_1 \iota_2 \iota_3 \iota_4} 
}}{\lr{l} \lr{m} \lr{n} \lr{p}} \wg^{(\iota_1)}(l) \wg^{(\iota_2)}(m) \wg^{(\iota_3)}(n) \wg^{(\iota_4)}(p) 
\mu_{\iota_1 \iota_2 \iota_3 \iota_4}^*
dlmnp.
\end{aligned}
\end{equation}
Hence, to prove Lemma \ref{lemma-quartic-term-bound} it suffices to show that bound 
in \eqref{quartic-term-bound} holds for every term in the right-hand side of the following inequality:
\begin{equation}
\label{integral-decomposition-in-N}
\begin{aligned}
\left\| \int_0^t \partial_k \left( e^{is\lr{k}} \distF \chi^4(s) \right) ds \right\|_{L_k^2} 
\lesssim &~ \normm{\int_0^t \partial_k \mathcal{N}_+(s,k) ds}{k} 
  + \normm{\int_0^t \partial_k \mathcal{N}_-(s,k) ds}{k}
  \\[3pt]
& + \normm{\int_0^t \partial_k \mathcal{N}_{R,1}(s,k) ds}{k} 
+ \normm{\int_0^t \partial_k \mathcal{N}_{R,2}(s,k) ds}{k}.
\end{aligned}
\end{equation}
In Subsection \ref{subsection-singular-N}, we prove the bound for the first 
two terms in \eqref{integral-decomposition-in-N} corresponding to the singular part. 
In Subsection \ref{subsection-regular-N}, we prove the bound for the rest of 
the terms corresponding to the regular part.

\smallskip
\subsection{Estimates for the singular part}
\label{subsection-singular-N}
Our goal here is to show that for every $t \geq 0$, \begin{equation}
\label{quartic-singular-part-goal}
\normm{\int_0^t \partial_k \mathcal{N}_\pm (s,k) ds}{k} \lesssim \|v\|_X^4.
\end{equation}
We will provide the proof for $\mathcal{N}_+$ only, as $\mathcal{N}_-$ case can be treated similarly.
We first expand $\psi_{S, +}$ functions in \eqref{spectral-singular-decomp} as the linear combination in \eqref{singular-part-pm-2}, and rewrite $\mu^{S, +}$ as follows
\begin{equation}
\label{mu-S-plus-1}
\begin{aligned}
\mu_{\iota_1 \iota_2 \iota_3 \iota_4}^{S, +}(k, l, m, n, p)= \sqrt{2\pi}
\sum_{\eps_0, \eps_1, \eps_2, \eps_3,\eps_4 \in \{+,-\}}& \overline{a_{+, \eps_0}(k)} ~ a_{+, \eps_1}^{(\iota_1)}(l) ~ a_{+, \eps_2}^{(\iota_2)}(m) ~ a_{+, \eps_3}^{(\iota_3)}(n) ~ a_{+, \eps_4}^{(\iota_4)}(p) \\
&\times \widehat{\nu_+} (\eps_0 k - \iota_1 \eps_1 l - \iota_2 \eps_2 m - \iota_3 \eps_3 n - \iota_4 \eps_4 p),
\end{aligned}
\end{equation}
where $\widehat{\nu_+}$ represents the (standard) Fourier transform of $\nu_+$. 
Then, substituting the expression \eqref{mu-S-plus-1} into \eqref{general-N-term}, 
and introducing the multilinear form
{\small \begin{equation}
\label{multilinear-form}
\begin{aligned}
\mathcal{T}_{\b, \eps_0, \eps_1, \eps_2, \eps_3, \eps_4}^{\iota_1, \iota_2, \iota_3, \iota_4} (f_1, f_2, f_3, f_4) (k) := ( \iota_1 \iota_2 \iota_3 \iota_4) &\int  \frac{e^{is \Phi_{\iota_1 \iota_2 \iota_3 \iota_4} 
}}{\lr{l} \lr{m} \lr{n} \lr{p}} f_1^{(\iota_1)} (l) f_2^{(\iota_2)} (m) f_3^{(\iota_3)} (n) f_4^{(\iota_4)} (p) 
\\[3pt]
&\times \b (\eps_0 k - \iota_1 \eps_1 l - \iota_2 \eps_2 m - \iota_3 \eps_3 n - \iota_4 \eps_4 p) dlmnp
\end{aligned}
\end{equation}}
with $\iota_j, \eps_j \in \{+,-\}$ and a distribution $\b$,
the formula for $\mathcal{N}_+$ is equivalent to
\begin{equation}
\label{N-plus-in-multilinear}
\begin{aligned}
4 \pi^2 \mathcal{N}_+(s,k) =  \sum_{\substack{\iota_1, \iota_2, \iota_3, \iota_4 \in \{+, -\} \\
\eps_0, \eps_1, \eps_2, \eps_3,\eps_4 \in \{+,-\}}}  \overline{a_{+, \eps_0}(k)} 
~
\mathcal{T}_{\widehat{\nu_+}, \eps_0, \eps_1, \eps_2, \eps_3, \eps_4}^{\iota_1, \iota_2, \iota_3, \iota_4} (a_{+, \eps_1} \wg, a_{+, \eps_2} \wg, a_{+, \eps_3} \wg, a_{+, \eps_4} \wg)(k) 
\end{aligned}
\end{equation}
To carry on the analysis of $\mathcal{N}_+$ 
we first analyze the properties of the multilinear form $\mathcal{T}$ in \eqref{multilinear-form}. 
For simplicity, we consider the particular form given by 
\begin{equation}
\label{particular-multilinear-form}
\mathcal{T}_{\b, \eps_0, \eps_1, \eps_2, \eps_3, \eps_4}(f_1, f_2, f_3, f_4)(k) 
  := \mathcal{T}_{\b, \eps_0, \eps_1, \eps_2, \eps_3, \eps_4}^{+,+,+,+} (f_1, f_2. f_3, f_4) (k)
\end{equation}
as for the remaining cases, when at least one of $
\iota_1, \iota_2, \iota_3, \iota_4$ is `$-$', 
the analysis is the same. 
We will also drop the dependence on $\epsilon_0$ when this is not relevant
(for example in Lemmas \ref{lemma-der-multilinear} and \ref{lemma-multilinear-in-physical}) 
with the understanding that we will be considering $\epsilon_0=+$.

The following algebraic lemma establishes an alternative representation 
of $\partial_k \mathcal{T}$, which will be used later to derive the desired estimate \eqref{quartic-singular-part-goal}.

\begin{lemma}
\label{lemma-der-multilinear}
For $f_j \in \mathcal{S}$, 
we have
\begin{equation}
\label{multilinear-form-derivative}
\begin{aligned}
\lr{k} \partial_k \mathcal{T}_{\b, \eps_1, \eps_2, \eps_3, \eps_4} (f_1, f_2, f_3, f_4)(k) =& ~i s \mathcal{T}_{q\b, \eps_1, \eps_2, \eps_3, \eps_4} (f_1, f_2, f_3, f_4)(k)\\[1.5pt]
& + \eps_1 \mathcal{T}_{\b, \eps_1, \eps_2, \eps_3, \eps_4} (\lr{\cdot} \partial_l f_1, f_2, f_3, f_4)(k)\\[1.5pt]
& + \eps_2 \mathcal{T}_{\b, \eps_1, \eps_2, \eps_3, \eps_4} (f_1, \lr{\cdot} \partial_m f_2, f_3, f_4)(k)\\[1.5pt]
& + \eps_3 \mathcal{T}_{\b, \eps_1, \eps_2, \eps_3, \eps_4} (f_1, f_2, \lr{\cdot} \partial_n  f_3, f_4)(k)\\[1.5pt]
& + \eps_4 \mathcal{T}_{\b, \eps_1, \eps_2, \eps_3, \eps_4} (f_1, f_2, f_3, \lr{\cdot} \partial_p f_4)(k)\\[2pt]
& + \int \frac{\Phi~ e^{is\Phi}}{\lr{l}\lr{m}\lr{n} \lr{p}} f_1(l) f_2(m) f_3(n) f_4(p) \partial_q b(q) dlmnp,
\end{aligned}
\end{equation}
where $\Phi := \Phi_{+,+,+,+}$, see \eqref{Phi-quartic}, 
and $q = k-\eps_1 l-\eps_2 m-\eps_3 n - \eps_4 p$. 
\end{lemma}

\begin{proof}
Throughout the proof we drop the $\eps_1, \eps_2,\eps_3,\eps_4$ indices for simplicity. 
Then, given $\eps_0 = \iota_1= \iota_2=\iota_3=\iota_4=+$ 
in \eqref{multilinear-form} from the choice of $\mathcal{T}_\b$, 
we differentiate \eqref{multilinear-form} with respect to $k$ and obtain 
\begin{equation}
\label{partial-k-T-1}
\lr{k} \partial_k \mathcal{T}_{\b} (f_1, f_2, f_3, f_4)(k) = is k \mathcal{T}_{\b} (f_1, f_2, f_3, f_4)(k) + \lr{k} \mathcal{T}_{\partial_q \b} (f_1, f_2, f_3, f_4)(k).
\end{equation}
Then, the expression \eqref{partial-k-T-1} can be further rewritten as
\begin{equation}
\label{partial-k-T-2}
\begin{aligned}
\lr{k} \partial_k \mathcal{T}_{\b} (f_1, f_2, f_3, f_4)(k) 
= & is \mathcal{T}_{q \b} (f_1, f_2, f_3, f_4)(k) + \lr{k} \mathcal{T}_{\partial_q \b} (f_1, f_2, f_3, f_4)(k)
\\
& + \eps_1 is \mathcal{T}_{\b} (\cdot f_1, f_2, f_3, f_4)(k) 
+ \eps_2is \mathcal{T}_{\b} (f_1, \cdot f_2, f_3, f_4)(k)\\
& + \eps_3is \mathcal{T}_{\b} (f_1, f_2, \cdot f_3, f_4)(k) + \eps_4 is\mathcal{T}_{\b} (f_1, f_2, f_3, \cdot f_4)(k),
\end{aligned}
\end{equation}
where $\cdot f$ denotes the multiplication of $f$ by its variable, $(\cdot f)(x) = x f(x)$.
Using $is l \lr{l}^{-1} e^{is \Phi} = - \partial_l (e^{is \Phi})$ 
and $q = k-\eps_1 l-\eps_2 m-\eps_3 n - \eps_4 p$, 
we integrate the third term in \eqref{partial-k-T-2} by parts in $l$ to obtain
\begin{equation}
\label{T-b-l-1}
\begin{aligned}
is\mathcal{T}_\b (\cdot f_1, f_2, f_3,f_4)(k) & 
= \int  \frac{is l \, e^{is \Phi}}{\lr{l} \lr{m} \lr{n} \lr{p}} 
f_1 (l) f_2 (m) f_3 (n) f_4 (p) \b (q) dlmnp
\\
&= \int  \frac{ e^{is \Phi}}{\lr{m} \lr{n} \lr{p}} [\partial_l f_1 (l)] f_2 (m) f_3 (n) f_4 (p)\b ( q) dlmnp\\
& \quad + \int  \frac{ e^{is \Phi}}{\lr{m} \lr{n} \lr{p}}  f_1 (l) f_2 (m) f_3 (n) f_4 (p)\partial_l \b ( q) dlmnp\\[3pt]
& = \mathcal{T}_\b (\lr{\cdot} \partial_l f_1, f_2, f_3,f_4)(k) -\epsilon_1 \mathcal{T}_{\partial_q \b} (\lr{\cdot} f_1, f_2, f_3,f_4)(k),
\end{aligned}
\end{equation}
where we used the relation $\partial_l \b (q) = -\eps_1 \partial_q \b (q)$. 
The last three terms in \eqref{partial-k-T-2} are manipulated similarly. 
Then, we substitute to the last four terms in \eqref{partial-k-T-2} 
their alternative expressions as in \eqref{T-b-l-1}, 
and obtain \eqref{multilinear-form-derivative}. 
This way the terms corresponding to $\mathcal{T}_{\partial_q \b}$ 
are combined into the last line in \eqref{multilinear-form-derivative}. 
\end{proof}

The following lemma will be used to transfer estimates
for the multilinear forms into a physical space. 

\begin{lemma}
\label{lemma-multilinear-in-physical}
Given $f_j \in \mathcal{S}$, we have 
\begin{equation}
\label{multi-linear-form-inverse-fourier}
\widehat{\mathcal{F}}^{-1} \left[ e^{-is\lr{k}}  \mathcal{T}_{\b, \eps_1, \eps_2, \eps_3, \eps_4} 
  (f_1, f_2, f_3, f_4)(k)\right] (x) = \widehat{\mathcal{F}}^{-1}[\b](x) \prod_{j =1,2,3,4} u_j(s, \eps_j x),
\end{equation}
where $u_j(s,x) = \lr{D}^{-1}e^{-is\lr{D}} \reallywidecheck{f_j}(x)$ with 
$\lr{D} = \sqrt{-\partial_x^2+1}$ and $\reallywidecheck{f}$ the inverse Fourier transform.
\end{lemma}

The proof relies on explicit computations, and is left to the reader.

A direct application of Lemma \ref{lemma-multilinear-in-physical} 
leads to 
\begin{equation}
\label{L-2-multilinear-singular-part-bound-1}
\normm{\mathcal{T}_{\b, \eps_1, \eps_2, \eps_3, \eps_4} (f_1, f_2, f_3, f_4)}{k} 
  \lesssim \Big\| \widehat{\mathcal{F}}^{-1}[\b] \prod_{j =1,2,3,4} u_j(s, \eps_j \cdot) \Big\|_{L^2_x}
\end{equation}

\begin{remark}
Throughout the computations for the singular part,
the distribution $\b$ is replaced by $\widehat{\nu_+}$,
which is given by the formula
\begin{equation}
\label{nu-plus-fourier}
\widehat{\nu_+}(k) = \sqrt{\frac{\pi}{2}}\delta_0 + \frac{\widehat{\zeta}(k)}{ik}+\widehat{\overline{\omega}} (k),
\end{equation}
where $\zeta$ is an even $C_c^\infty$ function with integral $1$ 
and $\overline{\omega}$ denotes a (generic) $C_c^\infty$ function.
See for example \cite{GPR18}
\end{remark}

From the formula \eqref{N-plus-in-multilinear} we see that $\mathcal{N}_+$ 
is a sum of several terms. 
The analysis for all these terms is similar and, therefore, 
we consider only the terms with fixed $\iota_1,\iota_2,\iota_3,\iota_4=+$
and drop those from our notation.
Then, to prove \eqref{quartic-singular-part-goal}, it suffices to show that,
for every $\epsilon_1,\epsilon_2,\epsilon_3,\epsilon_4\in\{+,-\}$ we have
\begin{equation}
\label{quartic-singular-part-goal-2}
\Big\| \int_0^t \partial_k 
  \sum_{\epsilon_0 \in \{+,-\}} \overline{a_{+, \eps_0}}(k) \,
  \mathcal{T}_{\widehat{\nu_+}, \eps_0, \eps_1, \eps_2, \eps_3, \eps_4} (f_1, f_2, f_3, f_4)(k) \, 
  ds \Big\|_{L^2_k} \lesssim \|v\|_X^4,
\end{equation}
where we set
\begin{equation*}
f_j(s,\cdot) = a_{+, \eps_j}(\cdot) \wg (\cdot) \quad \text{for  } j=1,2,3,4.
\end{equation*}

Notice that the coefficients $a_{+, \eps_0}(k)$ have jump discontinuities
at $k=0$, see \eqref{a-functions}, 
which may, in principle, result in a singular $\delta$-type contribution when differentiated.
However, using the formulas \eqref{a-functions} and the vanishing at $k=0$
of $T(k)$ and $R_\pm(k)+1$, see \eqref{T-behaviour-around-zero}, one can see that this is not the case
for the summation in \eqref{quartic-singular-part-goal-2}.
More precisely, we first write
\begin{align*}
& \sum_{\epsilon_0 \in \{+,-\}} \overline{a_{+, \eps_0}}(k) \,
  \mathcal{T}_{\widehat{\nu_+}, \eps_0, \eps_1, \eps_2, \eps_3, \eps_4} (f_1, f_2, f_3, f_4)
\\
& = \big[ \overline{a_{+,+}}(k) + \overline{a_{+,-}}(k) \big] \,
  \mathcal{T}_{\widehat{\nu_+}, +, \eps_1, \eps_2, \eps_3, \eps_4} (f_1, f_2, f_3, f_4)
\\ 
& + \overline{a_{+,-}}(k) \,
  \Big[ \mathcal{T}_{\widehat{\nu_+}, -, \eps_1, \eps_2, \eps_3, \eps_4} (f_1, f_2, f_3, f_4)
  - \mathcal{T}_{\widehat{\nu_+}, +, \eps_1, \eps_2, \eps_3, \eps_4} (f_1, f_2, f_3, f_4) \Big],
\end{align*}
and observe that the difference in the last line vanishes at $k=0$, 
see  \eqref{particular-multilinear-form} and \eqref{multilinear-form}.
Then, we use that, for $\iota_0 \in \{+,-\}$,
$\partial_k (a_{\iota_0, +}(k) + a_{\iota_0,-}(k)) = \mathbf{1}_{\iota_0}(k) \partial_k T(\iota_0k) 
  + \mathbf{1}_{-\iota_0}(k) \partial_k R_{\iota_0}(-\iota_0k)$.
Therefore, using also \eqref{uniform-estimate-T-R} and Minkowski's inequality we can bound
\begin{align*}
\Big\| \int_0^t \partial_k \sum_{\epsilon_0 \in \{+,-\}} \overline{a_{+, \eps_0}}
  ~ \mathcal{T}_{\widehat{\nu_+}, \eps_0, \eps_1, \eps_2, \eps_3, \eps_4} (f_1, f_2, f_3, f_4) ds \Big\|_{L^2_k} 
  \lesssim \int_0^t \normm{\mathcal{T}_{\widehat{\nu_+}, \eps_1, \eps_2, \eps_3, \eps_4} (f_1, f_2, f_3, f_4)}{k}ds 
  \\
  + \Big\|  \int_0^t 
  ~ \partial_k\mathcal{T}_{\widehat{\nu_+}, \eps_1, \eps_2, \eps_3, \eps_4} (f_1, f_2, f_3, f_4) ds \Big\|_{L^2_k}
:= E_1(t) + E_2(t).
\end{align*}
Below we give estimates for $E_1(t)$ and $E_2(t)$, 
and show that each of them is bounded by $\n{v}_X^4$, 
which would imply the bound \eqref{quartic-singular-part-goal}.


\subsubsection{Estimates for $E_1$}
We show that $E_1(t) \lesssim \n{v}_X^4$ for every $t \geq 0$. 
Using \eqref{L-2-multilinear-singular-part-bound-1}, 
\begin{align}\label{E1}
\begin{split}
E_1(t) & \lesssim \int_0^t \normm{\nu_+ u_1(s,\eps_1 \cdot) u_2(s,\eps_2 \cdot) u_3(s,\eps_3 \cdot) 
  u_4(s,\eps_4 \cdot)}{x} ds
  \\
& \lesssim \int_0^t \normm{u_1(s, \cdot)}{x} \|u_2(s, \cdot)\|_{L_x^\infty} \|u_3(s, \cdot)\|_{L_x^\infty} 
  \|u_4(s, \cdot)\|_{L_x^\infty} ds,
\end{split}
\end{align}
where $u_j(s,x) = \lr{D}^{-1} e^{-is \lr{D}} \widehat{\mathcal{F}}^{-1} (a_{+, \epsilon_j} \wg)$.
The 
formulas \eqref{a-functions} and the bounds in \eqref{uniform-estimate-T-R}
imply that for every $s \geq 0$ we have
\begin{equation}
\label{u-j-norm-for-E1}
\normm{u_j(s)}{x} 
  \lesssim \normm{\wg(s)}{k} \lesssim \n{v}_X.
\end{equation}
{Using \eqref{dispersive-estimates-standard-fourier} we can bound
\begin{align}\label{u-j-infty-norm}
\begin{split}
\norm{u_j(s)}{x}{\infty}
  \lesssim \lr{s}^{-1/2} \left( \big\| \partial_k \big( a_{+, \epsilon_j} \wg(s) \big) \big\|_{L^2_k} 
    + \|\lr{k}^2 
    \wg(s) \|_{L^2_k} \right);
\end{split}
\end{align}
using that $\wg(s,0)=0$, we see that no singularity arises when differentiating $a_{+,\eps_j}$
and, therefore,
\begin{align}\label{dkawtg}
\big\| \partial_k \big( a_{+, \epsilon_j} \wg(s) \big) \big\|_{L^2_k} 
  \lesssim \big\| \partial_k \wg(s) \big\|_{L^2_k};
\end{align}
then, from the definition \eqref{main-norm} and \eqref{u-j-infty-norm}-\eqref{dkawtg} we have
\begin{equation}
\label{u-j-infty-norm-for-E1}
\norm{u_j(s)}{x}{\infty} \lesssim \lr{s}^{-1/2} \n{v}_X.
\end{equation}
Using \eqref{u-j-norm-for-E1} and \eqref{u-j-infty-norm-for-E1} in \eqref{E1} we obtain}
\begin{equation*}
E_1(t) \lesssim \n{v}_X^4 \int_0^t \lr{s}^{-3/2} ds\lesssim \n{v}_X^4. 
\end{equation*}

\subsubsection{Estimates for $E_2$}
We want to show that 
for every $t \geq 0$,
\begin{equation}
\label{E2-bound-goal}
\normm{\int_0^t 
~ \partial_k\mathcal{T}_{\widehat{\nu_+}, \eps_1, \eps_2, \eps_3, \eps_4} 
  (f_1, f_2, f_3, f_4) ds}{k} \lesssim \n{v}_X^4.
\end{equation}
From Lemma \ref{lemma-der-multilinear}, the left-hand side of the above expression can be written as
\begin{equation}
\label{E2-decomposition}
\begin{aligned}
\Big\| \int_0^t 
~ \partial_k & \mathcal{T}_{\widehat{\nu_+}, \eps_1, \eps_2, \eps_3, \eps_4}  (f_1, f_2, f_3, f_4) ds \Big\|_{L_k^2}\\
& = \normm{\lr{k}^{-1} \int_0^t s
~ \mathcal{T}_{q\widehat{\nu_+}, \eps_1, \eps_2, \eps_3, \eps_4} (f_1, f_2, f_3, f_4) ds}{k}\\
&\quad  + \normm{\lr{k}^{-1} \int_0^t \mathcal{T}_{\widehat{\nu_+}, \eps_1, \eps_2, \eps_3, \eps_4} (\lr{\cdot}\partial_l f_1, f_2, f_3, f_4) ds}{k}\\
&\quad + \normm{\lr{k}^{-1} \int_0^t  \mathcal{T}_{\widehat{\nu_+}, \eps_1, \eps_2, \eps_3, \eps_4} (f_1, \lr{\cdot} \partial_m f_2, f_3, f_4) ds}{k}\\
& \quad + \normm{\lr{k}^{-1} \int_0^t  \mathcal{T}_{\widehat{\nu_+}, \eps_1, \eps_2, \eps_3, \eps_4} (f_1, f_2, \lr{\cdot} \partial_n f_3, f_4) ds}{k}\\
& \quad + \normm{\lr{k}^{-1} \int_0^t  \mathcal{T}_{\widehat{\nu_+}, \eps_1, \eps_2, \eps_3, \eps_4} (f_1, f_2, f_3, \lr{\cdot} \partial_p f_4) ds}{k}\\
& \quad + \normm{\lr{k}^{-1}\int_0^t \int \frac{\Phi~ e^{is\Phi}}{\lr{l}\lr{m}\lr{n} \lr{p}} f_1(l) f_2(m) f_3(n) f_4(p) \partial_q \widehat{\nu_+} (q) dlmnpds}{k},\\[3pt]
& =: E_2^{(1)}(t) + E_2^{(2)}(t) + E_2^{(3)}(t) + E_2^{(4)}(t) + E_2^{(5)}(t) + E_2^{(6)}(t),
\end{aligned}
\end{equation}
where we recall that $q = k -\eps_1 l - \eps_2 m - \eps_3 n - \eps_4 p$. 
Therefore, to prove \eqref{E2-bound-goal} it suffices to show that every $E_2^{(j)}$
in \eqref{E2-decomposition} is bounded by $\n{v}_X^4$. 
This will be established in the lemmas below.

\begin{lemma}
\label{lemma-E2-1-bound}
For every $t \geq 0$, we have 
\begin{equation}
\label{E2-1-bound}
E_2^{(1)}(t) \lesssim \n{v}_X^4.
\end{equation}
\end{lemma}

\begin{proof}
First, according to the Fourier representation of $\widehat{\nu_+}$ in \eqref{nu-plus-fourier}, we have 
\begin{equation*}
q \widehat{\nu_+} = -i \widehat{\zeta}(q) + q \widehat{\overline{\omega}}(q) = \widehat{\Psi}(q)
\end{equation*}
with $\Psi$ being a Schwartz function. 
In the estimates below we extract several $\lr{x}^{-2}$ 
factors from $\Psi$, and use them to apply the local $L_x^\infty$ decay estimates 
\eqref{local-l-infty-decay-standard-Fourier}. 
Then, using Minkowski inequality and \eqref{L-2-multilinear-singular-part-bound-1}, we have  
\begin{equation*} \begin{aligned} 
E_2^{(1)} & \lesssim \int_0^t s
\normm{\mathcal{T}_{q\widehat{\nu_+}, \eps_1, \eps_2, \eps_3, \eps_4} (f_1, f_2, f_3, f_4)}{k} ds\\
& \lesssim \int_0^t s \normm{\Psi u_1(s,\eps_1 \cdot) u_2(s,\eps_2 \cdot) u_3(s,\eps_3 \cdot) u_4(s,\eps_4 \cdot)}{x} ds\\
& \lesssim \int_0^t s  
\|\lr{x}^{-2}u_1(s)\|_{L_x^\infty} 
\|\lr{x}^{-2}u_2(s)\|_{L_x^\infty} \|\lr{x}^{-2}u_3(s)\|_{L_x^\infty} \|\lr{x}^{-2}u_4(s)\|_{L_x^\infty} ds,
\end{aligned}
\end{equation*}
where $u_j(s,x) = \lr{D}^{-1} e^{-is \lr{D}} \widehat{\mathcal{F}}^{-1} (a_{+, \epsilon_j} \wg)$ 
for every $j=1,2,3,4$. 
Applying the local decay estimates \eqref{local-l-infty-decay-standard-Fourier}, 
using \eqref{a-functions} and the bounds in \eqref{uniform-estimate-T-R}
together with $\widetilde{g}(0)=0$,
we get 
\begin{equation*}
E_2^{(1)} \lesssim \n{v}_X^4 \int_0^t s \lr{s}^{-3} ds \lesssim  \n{v}_X^4.
\end{equation*}

\end{proof}

\begin{lemma}
\label{lemma-E2-j-bound}
For every $t \geq 0$ and $j=2,3,4,5$, we have 
\begin{equation}
\label{E2-j-bound}
E_2^{(j)}(t) \lesssim \n{v}_X^4.
\end{equation}
\end{lemma}

\begin{proof}
The proofs for $j=2,3,4,5$ are identical so we only prove \eqref{E2-j-bound} for $j=2$. 
Using \eqref{L-2-multilinear-singular-part-bound-1} 
we can estimate as follows:
\begin{equation*}
\begin{aligned}
E_2^{(2)}(t) \lesssim \int_0^t \normm{u_1(s, \cdot)}{x} \|u_2(s, \cdot)\|_{L_x^\infty} 
  \|u_3(s, \cdot)\|_{L_x^\infty} \|u_4(s, \cdot)\|_{L_x^\infty} ds,
\end{aligned}
\end{equation*}
where 
\begin{equation*}
\left\{ 
\begin{array}{l}
u_1(s,x) = \lr{D}^{-1} e^{-is \lr{D}} \widehat{\mathcal{F}}^{-1} (\lr{\cdot} \partial_l [a_{+, \epsilon_1} \wg]),
\\
u_r(s,x) = \lr{D}^{-1} e^{-is \lr{D}} \widehat{\mathcal{F}}^{-1} (a_{+, \epsilon_r} \wg) \quad \text{ for  } r=2,3,4. 
\end{array}
\right.
\end{equation*}
For $u_1$ we have 
$\normm{u_1(s)}{x} \lesssim \n{\wg}_{H_l^1} \lesssim \n{v}_X$,
while the $L_x^\infty$ norms of $u_r$ for $r=2,3,4$ 
are bounded similarly to \eqref{u-j-infty-norm-for-E1}.
Combining these estimates, we get the desired result
\begin{equation*}
E_2^{(2)} \lesssim \n{v}_X^4 \int_0^t \lr{s}^{-3/2} ds \lesssim \n{v}_X^4.
\end{equation*}
\end{proof}

The next lemma will complete the proof of \eqref{E2-bound-goal}.

\begin{lemma}
\label{lemma-E2-6-bound}
For every $t \geq 0$, we have 
\begin{equation}
\label{E2-6-bound}
E_2^{(6)}(t) \lesssim \n{v}_X^4.
\end{equation}
\end{lemma}

\begin{proof}
We first consider the integral
\begin{equation*}
\mathcal{I} := \int_0^t \int \frac{\Phi~ e^{is\Phi}}{\lr{l}\lr{m}\lr{n} \lr{p}} f_1(l) f_2(m) f_3(n) f_4(p) \partial_q \widehat{\nu_+} (q) dlmnp~ds,
\end{equation*}
where $q = k -\eps_1 l - \eps_2 m -\eps_3 n -\eps_4 p$.
Using that $\Phi e^{is\Phi} = -i \partial_s e^{is\Phi}$, we integrate $\mathcal{I}$ by parts in $s$ to get 
\begin{equation}
\label{I-in-E2-6-decomposition}
\begin{aligned}
\mathcal{I} & = - i \int \frac{ e^{is\Phi}}{\lr{l}\lr{m}\lr{n} \lr{p}} f_1(l) f_2(m) f_3(n) f_4(p) \partial_q \widehat{\nu_+} (q) dlmnp \Bigg|_{s=0}^{s=t}\\
& \quad + i \int_0^t \int \frac{e^{is\Phi}}{\lr{l}\lr{m}\lr{n} \lr{p}} \partial_s [f_1(l) f_2(m) f_3(n) f_4(p)] \partial_q \widehat{\nu_+} (q) dlmnp ~ds\\[5pt]
&=: -i \mathcal{I}_1(t,k) + i \mathcal{I}_2(t,k).
\end{aligned}
\end{equation}
Both integrals contain the term $\partial_q \widehat{\nu_+} (q)$, which we treat using the relations 
\begin{equation}
\label{prtl-nu-plus}
\partial_q \widehat{\nu_+} (q) = - \eps_1 \partial_l \widehat{\nu_+} (q) 
  = - \eps_2 \partial_m \widehat{\nu_+} (q) = - \eps_3 \partial_n \widehat{\nu_+} (q) 
  = - \eps_4 \partial_p \widehat{\nu_+} (q)
\end{equation}
to integrate by parts and deal with $\widehat{\nu_+} (q)$ only. 

We look at 
$\mathcal{I}_1$ first. 
From \eqref{prtl-nu-plus}, we integrate by parts in $l$, and obtain 
\begin{equation}
\label{I1-in-E2-6-decomposition}
\begin{aligned}
\mathcal{I}_1 & =\eps_1 \int \frac{ e^{is\Phi}}{\lr{l}\lr{m}\lr{n} \lr{p}} \partial_l f_1(l) f_2(m) f_3(n) f_4(p) \widehat{\nu_+} (q) dlmnp \Bigg|_{s=0}^{s=t}\\
& \quad - \eps_1 \int \frac{l  e^{is\Phi}}{\lr{l}^3\lr{m}\lr{n} \lr{p}} f_1(l) f_2(m) f_3(n) f_4(p) \widehat{\nu_+} (q) dlmnp \Bigg|_{s=0}^{s=t}\\
&\quad - i \eps_1  \int \frac{s l  e^{is\Phi} }{\lr{l}^2\lr{m}\lr{n} \lr{p}} f_1(l) f_2(m) f_3(n) f_4(p) \widehat{\nu_+} (q) dlmnp \Bigg|_{s=0}^{s=t}\\[3pt]
&=: \mathcal{I}_1^{(1)}(t,k) + \mathcal{I}_1^{(2)}(t,k) +\mathcal{I}_1^{(3)}(t,k).
\end{aligned}
\end{equation}
The integrals above are similar in structure, and, in fact, can be treated in the same way. 
We only show how to deal with $\mathcal{I}_1^{(3)}$ as it contain the potentially 
problematic factor $s$. 
Also, we drop the endpoint values of $s$ and, abusing notations, only work with
\begin{equation*}
\mathcal{I}_1^{(3)}(s,k) := - i \eps_1  \int \frac{s l  e^{is\Phi} }{\lr{l}^2\lr{m}\lr{n} \lr{p}}
  f_1(l) f_2(m) f_3(n) f_4(p) \widehat{\nu_+} (q) dlmnp.
\end{equation*}
From \eqref{particular-multilinear-form}, we have 
\begin{equation*}
\mathcal{I}_1^{(3)}(s,k) = -i \eps_1 s \mathcal{T}_{ \widehat{\nu_+}, \eps_1, \eps_2, \eps_3, \eps_4} \left(\frac{\cdot}{\lr{\cdot}} f_1, f_2, f_3, f_4 \right).
\end{equation*}
Then, using \eqref{L-2-multilinear-singular-part-bound-1}, 
\begin{equation*}
\begin{aligned}
\normm{\mathcal{I}_1^{(3)}(s)}{k} & \lesssim s \normm{\nu_+ u_1(s, \eps_1\cdot) u_2(s, \eps_2 \cdot) u_3(s, \eps_3 \cdot) u_4(s, \eps_4 \cdot)}{x}\\
& \lesssim s \normm{u_1(s, \cdot)}{x} \|u_2(s, \cdot)\|_{L_x^\infty} \|u_3(s, \cdot)\|_{L_x^\infty} \|u_4(s, \cdot)\|_{L_x^\infty}
\end{aligned}
\end{equation*}
where 
\begin{equation*}
\left\{ 
\begin{array}{l}
u_1(s,x) = \lr{D}^{-1} e^{-is \lr{D}} \widehat{\mathcal{F}}^{-1} \left(\frac{\cdot}{\lr{\cdot}} a_{+, \epsilon_1} \wg\right),\\
u_r(s,x) = \lr{D}^{-1} e^{-is \lr{D}} \widehat{\mathcal{F}}^{-1} (a_{+, \epsilon_r} \wg) \quad \text{ for  } r=2,3,4.
\end{array}
\right.
\end{equation*}
Then, 
the dispersive decay estimates \eqref{dispersive-estimates-standard-fourier} imply
\begin{equation*}
\normm{\mathcal{I}_1^{(3)}(s)}{k} \lesssim s \lr{s}^{-3/2} \n{v}_X^4 \lesssim \n{v}_X^4.
\end{equation*}
Similar steps repeated for $\mathcal{I}_1^{(1)}(t,k)$ and $\mathcal{I}_1^{(2)}(t,k)$ 
from \eqref{I1-in-E2-6-decomposition} give the same bound $\n{v}_X^4$ for the integrals.
Combining such estimates we have $\normm{\mathcal{I}_1(t)}{k} \lesssim \n{v}_X^4$ for every $t\geq 0$, as desired.

It remains to verify $\normm{\mathcal{I}_2(t)}{k} \lesssim \n{v}_X^4$ for $\mathcal{I}_2$ 
given in \eqref{I-in-E2-6-decomposition}. As in previous computations, 
based on \eqref{prtl-nu-plus} we use $\partial_q \widehat{\nu_+}$ 
to integrate by parts in one of the integration variables $l, m, n$ or $p$. 
The choice of the variable 
depends on what function $\partial_s$ hits;
more precisely, 
when integrating by parts, we want 
avoid hitting the function of the type $\partial_s f_j$. 
Applying $\partial_s$ to the product of functions in the $\mathcal{I}_2$ 
expression in \eqref{I-in-E2-6-decomposition}, we obtain four terms which can all be treated 
in the same way, so we only look at one of them:
\begin{equation}
\label{I2-in-E2-6-decomposition}
\begin{aligned}
\mathcal{I}_2^{(1)}(t,k) 
  & := \int_0^t \int \frac{e^{is\Phi}}{\lr{l}\lr{m}\lr{n} \lr{p}} [\partial_s f_1(l)] 
  f_2(m) f_3(n) f_4(p) \partial_q \widehat{\nu_+} (q) dlmnp~ds
.
\end{aligned}
\end{equation}
We use \eqref{prtl-nu-plus} 
to integrate by parts in $m$, and using \eqref{particular-multilinear-form} we write
\begin{equation*}
\begin{aligned}
\mathcal{I}_2^{(1)}(t,k) & = \eps_2 \int_0^t \mathcal{T}_{ \widehat{\nu_+}, \eps_1, \eps_2, \eps_3, \eps_4} \left(\partial_s f_1, \partial_m f_2, f_3, f_4 \right) ds\\
& \quad - \eps_2 \int_0^t \mathcal{T}_{ \widehat{\nu_+}, \eps_1, \eps_2, \eps_3, \eps_4} \left(\partial_s f_1, \frac{\cdot}{\lr{\cdot}^2} f_2, f_3, f_4 \right) ds\\
& \quad - i \eps_2 \int_0^t s \mathcal{T}_{ \widehat{\nu_+}, \eps_1, \eps_2, \eps_3, \eps_4} \left(\partial_s f_1, \frac{\cdot}{\lr{\cdot}} f_2, f_3, f_4 \right) ds
=: J_1(t,k) + J_2(t,k) + J_3(t,k).
\end{aligned}
\end{equation*}
Note that all three integrals contain the term $\partial_s f_1$, 
and we eventually use it to get $\lr{s}^{-3/2}$ decay via \eqref{prtl-t-wg-estimates} 
which allows to overcome the factor $s$ in the integral $J_3$ without any difficulties. 

For the first integral we use 
\eqref{L-2-multilinear-singular-part-bound-1}, 
\begin{equation}
\label{J1-in-E2-6-bound-1}
\begin{aligned}
\normm{J_1(t)}{k} & \lesssim \int_0^t \normm{\nu_+ u_1(s, \eps_1\cdot) u_2(s, \eps_2 \cdot) u_3(s, \eps_3 \cdot) u_4(s, \eps_4 \cdot)}{x} ds\\
& \lesssim \int_0^t \|u_1(s, \cdot)\|_{L_x^\infty} \normm{u_2(s, \cdot)}{x}  \|u_3(s, \cdot)\|_{L_x^\infty} \|u_4(s, \cdot)\|_{L_x^\infty} ds
\end{aligned}
\end{equation}
where 
\begin{equation*}
\left\{ 
\begin{array}{l}
u_1(s,x) = \lr{D}^{-1} e^{-is \lr{D}} \widehat{\mathcal{F}}^{-1} (a_{+, \epsilon_1} \partial_s \wg),\\
u_2(s,x) = \lr{D}^{-1} e^{-is \lr{D}} \widehat{\mathcal{F}}^{-1} ( \partial_m [a_{+, \epsilon_2} \wg]),\\
u_r(s,x) = \lr{D}^{-1} e^{-is \lr{D}} \widehat{\mathcal{F}}^{-1} ( a_{+, \epsilon_r} \wg) \quad \text{ for  } r=3,4. 
\end{array}
\right.
\end{equation*}
To bound the $L_x^\infty$ norm of $u_1$ 
we use the simple relation $\norm{u_1(s)}{x}{\infty} \lesssim \normm{\lr{l} \widehat{u}_1}{l}$, 
which 
implies
\begin{equation}
\label{J1-u-for-prtl_s-bound}
\norm{u_1(s)}{x}{\infty} \lesssim \normm{\partial_s \wg}{l} \lesssim \lr{s}^{-3/2} \n{v}_X^2
\end{equation}
according to the bound in \eqref{prtl-t-wg-estimates}. 
From the dispersive decay estimates for the case of the standard Fourier transform 
\eqref{dispersive-estimates-standard-fourier},
\eqref{uniform-estimate-T-R}, and 
\eqref{dkawtg}, 
we have
\begin{equation}
\label{J1-other-u-bound}
\normm{u_2(s)}{x} \lesssim \n{v}_X \quad \text{and} \quad \|u_r(s)\|_{L_x^\infty}\lesssim \lr{s}^{-1/2} \n{v}_X.
\end{equation}
Using \eqref{J1-u-for-prtl_s-bound} and \eqref{J1-other-u-bound} into \eqref{J1-in-E2-6-bound-1}, we obtain
\begin{equation}
\label{J1-in-E2-6-bound}
\normm{J_1(t)}{k} \lesssim \n{v}_X^5 \int_0^t \lr{s}^{-5/2} ds \lesssim \n{v}_X^4. 
\end{equation}

The remaining integrals $J_2$ and $J_3$ can be treated in the same way with appropriate 
modifications for $u_2(s,x)$. 
In both cases, the $L_x^2$ norm of $u_2(s)$ is bounded by $\n{v}_X$,
and this lead to a $\n{v}_X^4$ bound for both $\normm{J_2(t)}{k}$ and $\normm{J_3(t)}{k}$, as desired. 
Combining these estimates, we obtain $\|\mathcal{I}_2(t)\|_{L^2_k} \lesssim \n{v}_X^4$.
Together with the previously obtained result for $\mathcal{I}_1$, 
we get the desired estimate \eqref{E2-6-bound}.
\end{proof}


\smallskip
\subsection{Estimates for the regular part}
\label{subsection-regular-N}
Our goal here is to show that for every $t \geq 0$, \begin{equation}
\label{quartic-regular-part-goal}
\normm{\int_0^t \partial_k \mathcal{N}_{R, j} (s,k) ds}{k} \lesssim \|v\|_X^4 \quad \text{for  } j = 1,2.
\end{equation}
We separate the analysis for $j=1$ and $j=2$. Recall the regular parts of the nonlinear spectral distribution given by \eqref{mu-r-1}--\eqref{mu-r-2} and the expression for $\mathcal{N}_{R,j}$ from \eqref{general-N-term}.

\subsubsection{Estimates for $\mathcal{N}_{R, 1}$}
Based on \eqref{general-N-term}, the expression for $\mathcal{N}_{R, 1}$ reads
\begin{equation}
\label{N-r-1-explicit}
\begin{aligned}
&(2 \pi)^{5/2}\mathcal{N}_{R,1}(s,k) \\
&= \sum_{\iota_1, \iota_2, \iota_3, \iota_4 \in \{+, -\}} \iota_1 \iota_2 \iota_3 \iota_4 \int \frac{e^{is \Phi_{\iota_1 \iota_2 \iota_3 \iota_4} 
}}{\lr{l} \lr{m} \lr{n} \lr{p}} \wg^{(\iota_1)}(l) \wg^{(\iota_2)}(m) \wg^{(\iota_3)}(n) \wg^{(\iota_4)}(p) 
\mu_{\iota_1 \iota_2 \iota_3 \iota_4}^{R,1}
dlmnp.
\end{aligned}
\end{equation}
The analysis for each term in the above sum is similar, 
therefore we may only consider the term with $\iota_1 = \iota_2 = \iota_3 = \iota_4 = +$. 
Next, we recall that $\mu_{++++}^{R,1}$ defined in \eqref{mu-r-1} is the sum of multiple integrals.
Every integral in the sum \eqref{mu-r-1} has at least one of the indices in $(A,B,C,D,E)$ equal to $R$;
they are similar in nature, and are treated essentially in the same way. 
In particular, it is sufficient to consider only two terms: 
\begin{itemize}
    \item the term with $A=R$;
    \item the term with $A=S, B=R$.
\end{itemize} 
For these terms, the analysis is independent on the rest of the indices.
so we do not specify them explicitly as $R$ and $S$. In other words, we consider 
\begin{equation}
\label{mu-r-1-1}
\mu^{R,1,1} := \int \overline{\psi_{R}(x,k)} \psi_{M_1}(x,l)
\psi_{M_2}(x,m) \psi_{M_3}(x,n)
\psi_{M_4}(x,p)dx, \quad M_j \in \{S, R\}, 
\end{equation}
and 
\begin{equation}
\label{mu-r-1-2}
\mu^{R,1,2} := \int \overline{\psi_{S}(x,k)} \psi_{M_1}(x,l)
\psi_{M_2}(x,m) \psi_{M_3}(x,n)
\psi_{R}(x,p)dx, \quad M_j \in \{S, R\}. 
\end{equation}
The corresponding integrals from \eqref{N-r-1-explicit} are 
\begin{equation*}
\begin{aligned}
\mathcal{N}_{R,1,1}(s,k) = \int \frac{e^{is \Phi
}}{\lr{l} \lr{m} \lr{n} \lr{p}} \wg(l) \wg(m) \wg(n) \wg(p) 
\mu^{R,1,1}
dlmnp,
\end{aligned}
\end{equation*}
and
\begin{equation*}
\begin{aligned}
\mathcal{N}_{R,1,2}(s,k) = \int \frac{e^{is \Phi
}}{\lr{l} \lr{m} \lr{n} \lr{p}} \wg(l) \wg(m) \wg(n) \wg(p) 
\mu^{R,1,2}
dlmnp,
\end{aligned}
\end{equation*}
where $\Phi = \Phi_{++++}$ given by \eqref{Phi-quartic}. 
After applying the $\partial_k$ differentiation to the above integral, we get for $j=1,2$ 
\begin{equation}
\label{mathcal-A-B}
\begin{aligned}
\partial_k \mathcal{N}_{R,1,j}(s,k) & = is \frac{k}{\lr{k}} \int \frac{e^{is \Phi
}}{\lr{l} \lr{m} \lr{n} \lr{p}} \wg(l) \wg(m) \wg(n) \wg(p) 
\mu^{R,1,j}
dlmnp\\
&\quad + \int \frac{e^{is \Phi
}}{\lr{l} \lr{m} \lr{n} \lr{p}} \wg(l) \wg(m) \wg(n) \wg(p) \partial_k
\mu^{R,1,j}
dlmnp\\
& =: \mathcal{A}_j(s,k) + \mathcal{B}_j(s,k). 
\end{aligned}
\end{equation}
To show the bound in \eqref{quartic-regular-part-goal} 
it suffices to verify the bound for the $L_k^2$ norms of $ds$ integral
associated with $\mathcal{A}_j$ and $\mathcal{B}_j$ from \eqref{mathcal-A-B}.
We show these estimates in the following lemmas.

\begin{lemma}
\label{lemma-A1-bound}
For every $t\geq 0$ we have 
\begin{equation*}
\normm{\int_0^t \mathcal{A}_1(s, k) ds}{k} \lesssim \n{v}_X^4.
\end{equation*}
\end{lemma}

\begin{proof}
We substitute $\mu^{R,1,1}$ from \eqref{mu-r-1-1} into the expression for $\mathcal{A}_1$, and expand the function $\Phi=\Phi_{++++}$ according to its definition in \eqref{Phi-quartic}. As a result, using \eqref{h-sing-reg-decomp}, we can rewrite $\mathcal{A}_1$ as
\begin{equation*}
\mathcal{A}_1(s,k) = 4i s \pi^2 \frac{k}{\lr{k}}e^{is \lr{k}} \int \overline{\psi_R(x,k)} u_{M_1}(x) u_{M_2}(x) u_{M_3}(x) u_{M_4}(x) dx,
\end{equation*}
where $u(x)= \lr{\wD}^{-1} e^{-is\lr{\wD}} g(s, x)$. 
Then, using the estimate \eqref{pseudo-psi-regular-bound}, we have 
\begin{equation*}
\begin{aligned}
\normm{\int_0^t \mathcal{A}_1(s, k) ds}{k} & \lesssim \int_0^t \normm{ \mathcal{A}_1(s, k)}{k} ds \lesssim \int_0^t s \normm{\lr{x}^{-11} u_{M_1} u_{M_2} u_{M_3} u_{M_4}}{x} ds\\
& \lesssim \int_0^t s \norm{\lr{x}^{-2} u_{M_1}}{x}{\infty} \norm{\lr{x}^{-2} u_{M_2}}{x}{\infty} \norm{\lr{x}^{-2} u_{M_3}}{x}{\infty} \norm{\lr{x}^{-2} u_{M_4}}{x}{\infty} ds.
\end{aligned}
\end{equation*}
From \eqref{local-l-infty-decay-singular} and \eqref{local-l-infty-decay-regular}, each of the above  $L_x^\infty$ norms is bounded by $\lr{s}^{-1} \n{\widetilde{g}}_{H_k^1}$, and this completes the proof.
\end{proof}

\begin{lemma}
\label{lemma-B1-bound}
For every $t\geq 0$ we have 
\begin{equation*}
\normm{\int_0^t \mathcal{B}_1(s, k) ds}{k} \lesssim \n{v}_X^4.
\end{equation*}
\end{lemma}

\begin{proof}
The proof is similar to the proof of Lemma \ref{lemma-A1-bound}. 
We rewrite $\mathcal{B}_1$ as
\begin{equation*}
\mathcal{B}_1(s,k) = 4 \pi^2 e^{is \lr{k}} \int \partial_k \overline{\psi_R(x,k)} u_{M_1}(x) u_{M_2}(x) u_{M_3}(x) u_{M_4}(x) dx,
\end{equation*}
where $u(x)= \lr{\wD}^{-1} e^{-is\lr{\wD}} g(s, x)$. Then the desired estimate follows from \eqref{pseudo-dk-psi-regular-bound}, \eqref{local-l-infty-decay-singular} and \eqref{local-l-infty-decay-regular}.
\end{proof}

\begin{lemma}
\label{lemma-A2-bound}
For every $t\geq 0$ we have 
\begin{equation*}
\normm{\int_0^t \mathcal{A}_2(s, k) ds}{k} \lesssim \n{v}_X^4.
\end{equation*}
\end{lemma}

\begin{proof}
Using \eqref{h-sing-reg-decomp}, we rewrite $\mathcal{A}_2$ as
\begin{equation*}
\mathcal{A}_2(s,k) = 4i s \pi^2 \frac{k}{\lr{k}}e^{is \lr{k}} \int \overline{\psi_S(x,k)} u_{M_1}(x) u_{M_2}(x) u_{M_3}(x) u_{R}(x) dx,
\end{equation*}
where $u(x)= \lr{\wD}^{-1} e^{-is\lr{\wD}} g(s, x)$. 
From the definition of $\psi_S$ in \eqref{psi-singular-part-decomposition}, one can verify that 
\begin{equation}
\label{psi-sing-transform-plancherel}
\normm{\int \overline{\psi_S(x,k)} h(x)~dx}{k} \lesssim \normm{h}{x}.
\end{equation}
Then, applying \eqref{psi-sing-transform-plancherel} and \eqref{local-l-infty-decay-regular} we get
\begin{equation*}
\begin{aligned}
\normm{\int_0^t \mathcal{A}_2(s, k) ds}{k} & \lesssim \int_0^t s \normm{u_{M_1} u_{M_2} u_{M_3} u_{R}}{x} ds\\
& \lesssim \int_0^t s \norm{\lr{x}^{-2} u_{M_1}}{x}{\infty} \norm{\lr{x}^{-2} u_{M_2}}{x}{\infty} \norm{\lr{x}^{-2} u_{M_3}}{x}{\infty} \norm{\lr{x}^{10} u_{R}}{x}{\infty} ds \lesssim \n{v}_X^4
\end{aligned}
\end{equation*}
as desired.
\end{proof}

\begin{lemma}
\label{lemma-B2-bound}
For every $t\geq 0$ we have 
\begin{equation*}
\normm{\int_0^t \mathcal{B}_2(s, k) ds}{k} \lesssim \n{v}_X^4.
\end{equation*}
\end{lemma}

\begin{proof}
The proof is similar to the proof of Lemma \ref{lemma-A2-bound}. 
We rewrite $\mathcal{B}_2$ as
\begin{equation*}
\mathcal{B}_2(s,k) = 4 \pi^2 e^{is \lr{k}} \int \partial_k \overline{\psi_S(x,k)} u_{M_1}(x) u_{M_2}(x) u_{M_3}(x) u_{R}(x) dx,
\end{equation*}
where $u(x)= \lr{\wD}^{-1} e^{-is\lr{\wD}} g(s, x)$. 
Similarly to \eqref{psi-sing-transform-plancherel}, one can verify
\begin{equation}
\label{dk-psi-sing-transform-plancherel}
\normm{\int \partial_k \overline{ \psi_S(x,k)} h(x)~dx}{k} \lesssim \normm{\lr{x}h}{x}.
\end{equation}
Then, we obtain
\begin{equation*}
\begin{aligned}
\normm{\int_0^t \mathcal{B}_2(s, k) ds}{k} & \lesssim \int_0^t s \normm{\lr{x} u_{M_1} u_{M_2} u_{M_3} u_{R}}{x} ds\\
& \lesssim \int_0^t s \norm{\lr{x}^{-2} u_{M_1}}{x}{\infty} \norm{\lr{x}^{-2} u_{M_2}}{x}{\infty} \norm{\lr{x}^{-2} u_{M_3}}{x}{\infty} \norm{\lr{x}^{10} u_{R}}{x}{\infty} ds \lesssim \n{v}_X^4.
\end{aligned}
\end{equation*}
\end{proof}
This completes the proof of \eqref{quartic-regular-part-goal} for $j=1$.


\smallskip
\subsubsection{Estimates for $\mathcal{N}_{R, 2}$}
Based on \eqref{general-N-term}, the expression for $\mathcal{N}_{R, 2}$ reads
\begin{equation}
\label{N-r-2-explicit}
\begin{aligned}
&(2 \pi)^{5/2}\mathcal{N}_{R,2}(s,k) \\
&= \sum_{\iota_1, \iota_2, \iota_3, \iota_4 \in \{+, -\}} \iota_1 \iota_2 \iota_3 \iota_4 \int \frac{e^{is \Phi_{\iota_1 \iota_2 \iota_3 \iota_4} 
}}{\lr{l} \lr{m} \lr{n} \lr{p}} \wg^{(\iota_1)}(l) \wg^{(\iota_2)}(m) \wg^{(\iota_3)}(n) \wg^{(\iota_4)}(p) 
\mu_{\iota_1 \iota_2 \iota_3 \iota_4}^{R,2}
dlmnp.
\end{aligned}
\end{equation}
As previously, we only consider the contribution from $\mu_{++++}^{R,2}$. 
Then, using \eqref{psi-singular-part-decomposition}, we can rewrite \eqref{mu-r-2} as
\begin{equation*}
\begin{aligned}
&\mu_{++++}^{R,2}(k,l,m,n,p)
\\ 
& = \sum \int \sigma_{\eps_0}(x) \overline{\psi_{S, \eps_0}}(x,k) \sigma_{\eps_1}(x)
  \psi_{S, \eps_1}(x,l) \sigma_{\eps_2}(x)\psi_{S, \eps_2}(x,m) \sigma_{\eps_3}(x)
  \psi_{S, \eps_3}(x,n) \sigma_{\eps_4}(x)\psi_{S, \eps_4}(x,p) dx,
\end{aligned}
\end{equation*}
where the sum is over 
\begin{equation*}
\eps_0, \eps_1, \eps_2, \eps_3, \eps_4\in \{+,-\} \quad \text{with   }  (\eps_0, \eps_1, \eps_2, \eps_3, \eps_4) \neq (+,+,+,+,+), (-,-,-,-,-).
\end{equation*}
Note that the integrand in each integral is compactly supported due 
to the presence of a $\sigma_+(x)\sigma_-(x)$ factor. 
The structure of the integrals is similar, and, therefore, 
it suffices to consider only the measure 
\begin{equation}
\label{mu-r-2-1}
\mu^{R,2,1}(k,l,m,n,p) := \int \sigma_{+}(x) \overline{\psi_{S, +}(x,k)} \sigma_{-}(x)\psi_{S, -}(x,l) \psi_{S}(x,m) \psi_{S}(x,n) \psi_{S}(x,p) dx,
\end{equation}
and the corresponding integral
\begin{equation*}
\begin{aligned}
\mathcal{N}_{R,2,1}(s,k) = \int \frac{e^{is \Phi
}}{\lr{l} \lr{m} \lr{n} \lr{p}} \wg(l) \wg(m) \wg(n) \wg(p) 
\mu^{R,2,1}(k,l,m,n,p)
dlmnp \qquad \text{with   } \Phi= \Phi_{++++}.
\end{aligned}
\end{equation*}
Using \eqref{h-sing-reg-decomp} and \eqref{mu-r-2-1}, we can rewrite 
\begin{equation}
\label{N-r-2-1}
\mathcal{N}_{R,2,1}(s,k) = 4\pi^2 e^{is\lr{k}} \int \sigma_+(x) \sigma_-(x) 
  \overline{\psi_{S,+}(x,k)} u_-(x) u_S^3(x) dx,
\end{equation}
where $u(x)= \lr{\wD}^{-1} e^{-is\lr{\wD}} g(s, x)$ and $u_-$ stands for
\begin{equation*}
u_-(x) := \frac{1}{\sqrt{2\pi}} \int_\R \psi_{S,-}(x,k) \widetilde{u}(k) dk.
\end{equation*}
We apply the $\partial_k$ differentiation to \eqref{N-r-2-1} and obtain
\begin{equation*}
\begin{aligned}
\partial_k \mathcal{N}_{R,2,1}(s,k) &= 4is \pi^2 \frac{k}{\lr{k}} e^{is\lr{k}}
  \int \sigma_+(x) \sigma_-(x) \, \overline{\psi_{S,+}(x,k)} u_-(x) u_S^3(x) dx
  \\
&\quad + 4\pi^2 e^{is\lr{k}} \int \sigma_+(x) \sigma_-(x) \partial_k 
  \,\overline{\psi_{S,+}(x,k)} u_-(x) u_S^3(x) dx
\\
& = : \mathcal{C}_1 + \mathcal{C}_2.
\end{aligned}
\end{equation*}
Both terms $\mathcal{C}_1$ and $\mathcal{C}_2$ are treated using 
the presence of the compactly supported function $\sigma_+(x) \sigma_-(x)$
and the local decay estimates . 
We also observe that, similarly to \eqref{psi-sing-transform-plancherel} 
and \eqref{dk-psi-sing-transform-plancherel}, we have 
\begin{equation*}
\normm{\int \partial_k^r \overline{\psi_{S,+}(x,k)} h(x)~dx}{k} \lesssim \normm{\lr{x}^r h}{x}, \qquad r =0,1.
\end{equation*}
Then, using \eqref{local-l-infty-decay-singular} and 
an analogous \eqref{local-l-infty-decay-singular}-type estimate 
for $u_-$, we have
\begin{equation*}
\begin{aligned}
\normm{\int_0^t \mathcal{C}_1(s, k) ds}{k} & \lesssim \int_0^t s \normm{\sigma_- \sigma_+ u_{-} u_S^3}{x} ds
\\
& \lesssim \int_0^t s \norm{\lr{x}^{-2} u_-}{x}{\infty} \norm{\lr{x}^{-2} u_{S}}{x}{\infty}^3 ds \lesssim \n{v}_X^4.
\end{aligned}
\end{equation*}
Similar steps ensure that $\normm{\int_0^t \mathcal{C}_2(s, k) ds}{k} \lesssim \n{v}_X^4$. 
Combining the above estimates, we get 
\begin{equation*}
\normm{\int_0^t \partial_k \mathcal{N}_{R,2,1}(s, k) ds}{k} \lesssim \n{v}_X^4
\end{equation*}
which completes the proof \eqref{quartic-regular-part-goal} for $j=2$.

\subsubsection{Proof of Lemma \ref{lemma-quartic-term-bound}}
The proof follows from \eqref{integral-decomposition-in-N} and the estimates 
\eqref{quartic-singular-part-goal} and \eqref{quartic-regular-part-goal}.

\end{document}